\documentclass[10pt]{article}
\usepackage{amsfonts}
\usepackage{amssymb}
\usepackage{amsmath}
\usepackage{graphicx}
\usepackage{t1enc}
\usepackage{microtype} 
\usepackage{titling}
\usepackage[latin1]{inputenc}
\usepackage[english]{babel}
\usepackage[T1]{fontenc}
\usepackage[all]{xy}
\usepackage{latexsym}
\usepackage{t1enc}
\usepackage{epsfig,pslatex,colortbl}
\usepackage[hmarginratio=1:1,top=32mm,columnsep=20pt]{geometry} 
\usepackage{hyperref}
\usepackage{amsthm}
\usepackage{url}

\usepackage{array,multirow,makecell}
\setcellgapes{1pt}
\makegapedcells
\newcolumntype{R}[1]{>{\raggedleft\arraybackslash }b{#1}}
\newcolumntype{L}[1]{>{\raggedright\arraybackslash }b{#1}}
\newcolumntype{C}[1]{>{\centering\arraybackslash }b{#1}}
\newcolumntype{t}[1]{>{\arraybackslash }b{#1}}

\newcommand {\junk}[1]{}

\newtheorem{thm}{Theorem}[section]
\newtheorem{cor}[thm]{Corollary}
\newtheorem{lem}[thm]{Lemma}
\newtheorem{prop}[thm]{Proposition}

\newtheorem{defn}[thm]{Definition}

\newtheorem{theorem}{Theorem}[section]

\newtheorem{proposition}[theorem]{Proposition}

\newtheorem{example}[theorem]{Example}

\let \bb=\mathbb{}

\def\rm{\textrm{}}
\def\fudge{\mathchoice{}{}{\mkern.5mu}{\mkern.8mu}}
\def\bbc#1#2{{\rm \mkern#2mu\vbar\mkern-#2mu#1}}
\def\bbb#1{{\rm I\mkern-3.5mu #1}}
\def\bba#1#2{{\rm #1\mkern-#2mu\fudge #1}}
\def\bb#1{{\count4=`#1 \advance\count4by-64 \ifcase\count4\or\bba
A{11.5}\or \bbb B\or\bbc C{5}\or\bbb D\or\bbb E\or\bbb F \or\bbc
G{5}\or\bbb H\or \bbb I\or\bbc J{3}\or\bbb K\or\bbb L \or\bbb
M\or\bbb N\or\bbc O{5} \or \bbb P\or\bbc C{5}\or\bbb B\or\bbc
S{4.2}\or\bba T{10.5}\or\bbc U{5}\or \bba V{12}\or\bba
W{16.5}\or\bba X{11}\or\bba Y{11.7}\or\bba Z{7.5}\fi}}

\usepackage{amssymb}
\usepackage{relsize}
\usepackage{layout}

\begin{document}
\title{\vspace{-15mm}\fontsize{17pt}{10pt}\selectfont\textbf Cohomology and Deformations
of $n$-Lie algebra 
morphisms} 
\author{
\textsc{Anja Arfa\thanks{arfaanja.mail@gmail.com}}\\[2mm] 
\normalsize Universit\'e de Sfax (Tunisia) \\
\normalsize Facult\'{e} des sciences \\ 
\vspace{-5mm}
\and
\textsc{Nizar Ben Fraj\thanks{benfraj\_nizar@yahoo.fr}}\\[2mm] 
\normalsize Universit\'{e} de Carthage (Tunisia)\\
\normalsize Institut pr\'eparatoire aux \'etudes d'ing\'enieur de Nabeul\\ 
\vspace{-5mm}
\and
\textsc{Abdenacer Makhlouf\thanks{abdenacer.makhlouf@uha.fr}}\\[2mm] 
\normalsize Universit\'e de Haute Alsace (France) \\
\normalsize IRIMAS - d\'epartement de math\'ematiques\\ 
\vspace{-5mm}
}
\maketitle

\begin{abstract}
The study of $n$-Lie algebras which are natural generalization of Lie algebras is motivated by Nambu Mechanics and recent developments in String Theory and M-branes.  
The purpose of this paper is to define  cohomology complexes and study  deformation theory
of $n$-Lie algebra morphisms.
We discuss  infinitesimal deformations, equivalent deformations and obstructions. Moreover, we provide various examples.
\end{abstract}
\bigskip
\thispagestyle{empty}

\section*{Introduction}

Ternary operations and more generally $n$-ary operations appeared for the first time associated with the cubic matrices studied  by A. Cayley in  the XIXth century. Since then, multioperators rings and algebras have been studied in various modern mathematical works. The interest on  generalization of ordinary Lie algebras was motivated by their connection to Nambu Mechanics \cite{Y} which allows to consider more than one hamiltonian and also by more recent applications in String Theory and M-branes. For more application in Physics see  \cite{SO}.
The algebraic study of $n$-Lie algebras or $n$-ary Nambu-Lie
algebras was introduced first  by Filippov in  \cite{V} and completed by Kasymov in \cite{S}.
A $n$-Lie algebra is defined by a $n$-ary multilinear operation which is skew-symmetric and satisfies Filippov-Jacobi identity. In particular, for $n=3$ this identity is
$$[x_{1},x_{2},[x_{3},x_{4},x_{5}]]=[[x_{1},x_{2},x_{3}],x_{4},x_{5}]+[x_{3},[x_{1},x_{2},x_{4}],x_{5}]+
[x_{3},x_{4},[x_{1},x_{4},[x_{1},x_{2},x_{5}]].$$

Many proprieties on these types of algebras are treated,
one cites for example solvability, nilpotency, central extension.
A study of $(n+1)$-Lie algebras constructed
from $n$-Lie algebras and generalized trace maps was discussed in \cite{AMSS}. In \cite{KA}, the authors established the relationships
between the properties of a $n$-Lie algebra and its induced $(n+1)$-Lie algebra.

This aim of this paper is to construct a cohomology complex  of $n$-Lie 
algebra  morphisms. For that, we define a cohomology
structure of these algebras with values in a module.
The cohomology of $n$-Lie algebras is a generalization of the Chevalley-Eilenberg Lie algebras cohomology.
We  refer to \cite{J,GP,T} for the cohomology of $n$-Lie
algebras and for representation theory which leads us to describe the deformation cohomology
complex $(C^{*}(\phi,\phi),\delta)$ for $n$-Lie algebra morphisms. Notice that  Daletskii and Takhtajan showed, in \cite{DT}, that the cohomology
of $n$-Lie algebras can be derived from Leibniz algebras cohomology.
Furthermore, we are interested  in this paper by  the deformations of  $n$-Lie 
algebra morphisms.
A cohomology complex adapted to the study of deformations was introduced by Fr\'{e}gier \cite{f} in the case of Lie algebra morphisms.
Deformations of $n$-Lie algebras has been discussed in terms of Chevalley-Eilenberg cohomology in various papers, a survey is available  in \cite{A}. In this paper, we introduce and study for $n$-Lie algebra morphisms  
 the concepts of  infinitesimal and equivalent deformations as well as obstructions to extend a given fixed order deformation. We denote by $\mathcal{N}$ and $\mathcal{N}'$  two $n$-Lie algebras.
Equivalence classes of infinitesimal deformations of $n$-Lie 
algebras are characterized by the
cohomology groups $H^{2}(\mathcal{N},\mathcal{N})$ and by $H^{1}(\mathcal{N},\mathcal{N}')$ withe respect to 
a  morphism $\phi$ between $\mathcal{N}$ and $\mathcal{N}'$.

The paper is organized as follows. In Section 1, we review some basics
about $n$-Lie algebras and  their adjoint representation via a morphism $\phi$. Moreover, we recall the cohomology of $n$-Lie algebras with values in the algebra.
In Section $2$, we define the cohomology of $n$-Lie algebras with values in an adjoint module.
In Section $3$, we define explicitly a cochain complex with a coboundary operator  and the $n$-cochains
module $C^{n}(\phi,\phi)$ providing a cohomology of $n$-Lie algebra morphisms. Section $4$ focuses on the deformation theory of $n$-Lie algebra morphisms.
Lastly, in Section $5$, we deal with some examples. We compute the cohomology and provide  examples of deformation of 3-Lie algebra morphisms.
\section{Basics}
In this  section, we summarize the main definitions about $n$-Lie algebras, representations and cohomology of $n$-Lie algebras with values in the algebra itself. \\ In the sequel all vector spaces are considered over an algebraically closed field $\mathbb{K}$ of characteristic 0. 

\begin{defn}
A $n$-Lie
algebra is a  vector space $\mathcal{N}$ together with a $n$-ary multilinear operation
$[\cdot,\ldots,\cdot]$ satisfying the following identities:
\begin{eqnarray}\label{eq1}
[x_{1},\ldots,x_{n}]&=&(-1)^{\tau(\sigma)}[x_{\sigma(1)},\ldots,x_{\sigma(n)}],\\
\label{eq2}
 [x_{1},\ldots,x_{n-1},[y_{1},\ldots,y_{n}]]&=&\sum\limits_{i=1}^{n}[y_{1},\ldots,y_{i-1},[x_{1},\ldots,x_{n-1},y_{i}],y_{i+1},\ldots,y_{n}],
\end{eqnarray}
\noindent where $\sigma$ runs over the symmetric group $S_{n}$ and the number $\tau(\sigma)$ equals $0$ or $1$ depending on the
parity of the permutation $\sigma$.
We call condition \eqref{eq2} Nambu identity, it is also called fundamental identity or Filippov identity.
\end{defn}
The map
\begin{equation}\label{eq3}
  ad(x_{1},\ldots,x_{n-1})(x_{n})=[x_{1},\ldots,x_{n-1},x_{n}]\quad for~x_{n}\in \mathcal{N}
\end{equation}
is referred  as a left multiplication defined by elements $x_{1},\ldots,x_{n-1}\in \mathcal{N}$.
The Nambu  identity \eqref{eq2} may be written as 
$$ad(x_{1},\ldots,x_{n-1})([y_{1},\ldots,y_{n}])=\sum\limits_{i=1}^{n}[y_{1},\ldots,y_{i-1},ad(x_{1},\ldots,x_{n-1})(y_{i}),y_{i+1},\ldots,y_{n}].$$
The identity \eqref{eq2}, for $n=2$, corresponds to the Jacobi identity.
\begin{defn}
Let $(\mathcal{N},[\cdot,\ldots,\cdot])$ and $(\mathcal{N}',[\cdot,\ldots,\cdot]')$ be two $n$-Lie algebras.
A linear map $f:\mathcal{N}\rightarrow\mathcal{N}'$ is a $n$-Lie algebra morphism if it satisfies
$$f([x_{1},\ldots,x_{n}])=[f(x_{1}),\ldots,f(x_{n})]'.$$
\end{defn}
The concept of representation of Lie algebras is generalized to $n$-Lie algebras in a natural was as follows.
\begin{defn}
A representation of a $n$-Lie algebra $(\mathcal{N},[\cdot,\ldots,\cdot])$ on a vector space
$V$ is a skew-symmetric multilinear map $\rho:\mathcal{N}^{n-1}\rightarrow End(V)$ satisfying
the identities
\begin{eqnarray}\label{eq4}
&\rho(x_{1},\ldots,x_{n-1})\circ\rho(y_{1},\ldots,y_{n-1})-\rho(y_{1},\ldots,y_{n-1})\circ\rho(x_{1},\ldots,x_{n-1})
=\sum\limits_{i=1}^{n-1}\rho(y_{1},\ldots,y_{i-1},[x_{1},\ldots,x_{n-1},y_{i}],y_{i+1},\ldots,y_{n-1}),\\
&\rho([x_{1},\ldots,x_{n-1},x_{n}],y_{2},\ldots,y_{n-1})
=\sum\limits_{i=1}^{n}\rho(x_{1},\ldots,\widehat{x}_{i},\ldots,x_{n})\circ\rho(x_{i},y_{2},\ldots,y_{n-1}),
\end{eqnarray}
for all $x_{i},y_{j}\in \mathcal{N}, 1\leq i \leq n, 1\leq j \leq n-1$.
\end{defn}
\begin{example}
Let $(\mathcal{N},[\cdot,\ldots,\cdot])$ be a $n$-Lie algebra. The map $ad$ defined in \eqref{eq3} is a
representation. It is called adjoint representation.
\end{example}

Let $(\mathcal{N},[\cdot,\ldots,\cdot])$,  $(\mathcal{N}',[\cdot,\ldots,\cdot]')$ be two
$n$-Lie algebras and $\phi:\mathcal{N}\rightarrow \mathcal{N}'$ be a $n$-Lie algebra morphism.
Let $\wedge^{n-1}\mathcal{N}$ be the set of elements $x_{1}\wedge\cdots \wedge x_{n-1}$ that are skew-symmetric
in their arguments. 
On $\wedge^{n-1}\mathcal{N}$, for $x=x_{1}\wedge\cdots \wedge x_{n-1}\in \wedge^{n-1}\mathcal{N}$, $y=y_{1}\wedge\cdots\wedge y_{n-1}\in\wedge^{n-1}\mathcal{N}, ~z\in  \mathcal{N}'$, we define:
\begin{itemize}
  \item The linear map $L':\wedge^{n-1}\mathcal{N}\wedge  \mathcal{N}'\rightarrow  \mathcal{N}'$~ by~
$L'(x).z=[\phi(x_{1}),\ldots,\phi(x_{n-1}),z]'$.  
  \item The bilinear map $[~,~]:\wedge^{n-1}\mathcal{N}\times\wedge^{n-1}\mathcal{N}\rightarrow \wedge^{n-1}\mathcal{N}$
~by~ $[x,y]=L(x)\bullet y=\sum\limits_{i=1}^{n-1}(y_{1},\ldots,ad(x).y_{i},\ldots,y_{n-1})$.
\item The map $\bar{\phi}:\wedge^{n-1}\mathcal{N}\rightarrow\wedge^{n-1}\mathcal{N}$~ by~ 
$\bar{\phi}(x)=\phi(x_1)\wedge\ldots\wedge\phi(x_{n-1})$.
\end{itemize}
We denote by $\mathcal{L}(\mathcal{N})$ the space $ \wedge^{n-1}\mathcal{N}$ and we call it the fundamental set.
\begin{lem}\label{lem1}


The map $L'$ satisfies
\begin{equation}\label{lem2}
  L'([x,y]).z=L'(x).( L'(y).z)-L'(y).( L'(x).z)
\end{equation}
for all $x,y\in\mathcal{L}(\mathcal{N})$, $z\in \mathcal{N}'$. Then, $\mathcal{N}'$ is a $\mathcal{N}$-module called adjoint representation of $\mathcal{N}$ via $\phi$.
\end{lem}
\begin{proof} 
We have
\begin{eqnarray*}
  L'([x,y])\cdot z&=&L'\left(\sum\limits_{i=1}^{n-1}(y_{1},\ldots,ad(x).y_{i},\ldots,y_{n-1})\right)\cdot z\\
  &=&\sum\limits_{i=1}^{n-1}
  [\phi(y_{1}),\ldots,\phi\circ ad(x).y_{i},\ldots,\phi(y_{n-1}),z]'\\
  &=&\sum\limits_{i=1}^{n-1}[\phi(y_{1}),\ldots,\phi\circ[x_{1},\ldots,x_{n-1},y_{i}],\ldots,\phi(y_{n-1}),z]'.\\
  L'(x).(L'(y).z)-L'(y).(L'(x).z)&=&\sum\limits_{i=1}^{n-1}[\phi(y_{1}),\ldots,[\phi(x_{1}),\ldots,\phi(x_{n-1}),\phi(y_{i})]'
  ,\ldots,\phi(y_{n-1}),z]'\\
  &=&\sum\limits_{i=1}^{n-1}[\phi(y_{1}),\ldots,\phi\circ([x_{1},\ldots,x_{n-1},y_{i}])
  ,\ldots,\phi(y_{n-1}),z]'.
\end{eqnarray*}
Thus, the result holds.
\end{proof}
Moreover, we have the following fundamental result, providing a representation of a $n$-Lie algebra by a Leibniz algebra. Recall that a Leibniz algebra is a vector space with a binary bracket satisfying the following identity.
\begin{equation}\label{LeibnizIdentity}
[[X, Y], Z]= [[X, Z], Y] + [X, [Y, Z]]. 
\end{equation}
\begin{proposition}
Let $(\mathcal{N}, [\cdot,\cdots,\cdot])$ be a $n$-Lie algebra, then $\mathcal{L}(\mathcal{N})= (\mathcal{L}(\mathcal{N}), [\cdot,\cdot ] )$ is a Leibniz algebra with respect to the bracket
$$[x_1, \cdots, x_{n-1}, y_1, \cdots, y_{n-1}]= \sum_{i=1}^{n-1} x_1 \wedge \cdots \wedge [x_i, y_1, \cdots, y_{n-1}]\wedge \cdots \wedge x_{n-1},$$
for all $X=(x_1, \cdots, x_{n-1})$ and $Y= (y_1, \cdots, y_{n-1})$ in $\mathcal{L}(\mathcal{N})$.
\end{proposition}
 
Notice that  $ \wedge^{n-1}\mathcal{N}$ merely reflects that the fundamental object $X=(x_1, \cdots, x_{n-1})\in \wedge^{n-1}\mathcal{N}$
is antisymmetric in its arguments; it does not imply that $X$ is a $(n-1)$-multivector obtained by the associative.

Now, we recall the cochain complex  of $n$-Lie algebras with values in itself; for more details see
%
\cite{T, J}.
\begin{defn}
Let $(\mathcal{N},[\cdot ,\ldots,\cdot ])$ be a $n$-Lie algebra, a $\mathcal{N}$-valued $(p+1)$-cochain is a linear map
$\psi:\otimes^{p}\mathcal{L}(\mathcal{N})\wedge\mathcal{N}\rightarrow \mathcal{N}$. We denote by $C^{p}(\mathcal{N},\mathcal{N})$
the set of the $(p+1)$-cochains.
The coboundary operator  $\delta^{p+1}:C^{p}(\mathcal{N},\mathcal{N})\rightarrow C^{p+1}(\mathcal{N},\mathcal{N})$
is a linear map defined for $\psi\in C^{p}(\mathcal{N},\mathcal{N})$ by
\begin{equation}
\label{cohomology}
\begin{array}{lllll}
\delta^{p+1}\psi(a_{1},\ldots,a_{p},a_{p+1},z)&=&\sum\limits_{1\leq i< j\leq p+1}(-1)^{i}\psi(a_{1},\ldots,\widehat{a}_{i},\ldots,
a_{j-1},[a_{i},a_{j}],\ldots,a_{p+1},z)\\[0.5pt]
&+&\sum\limits_{i=1}^{p+1}(-1)^{i}\psi(a_{1},\ldots,\widehat{a}_{i},\ldots,a_{p+1},ad(a_{i}).z)\\[0.5pt]
&+&
\sum\limits_{i=1}^{p+1}(-1)^{i+1}ad(a_{i}).\psi(a_{1},\ldots,\widehat{a}_{i},\ldots,a_{p+1},z)\\[0.5pt]
&+&(-1)^{p}\sum\limits_{i=1}^{n-1}
[a_{p+1}^1,\ldots,\psi(a_{1},\ldots,a_{p},a_{p+1}^{i}),\ldots,a_{p+1}^{n-1},z]
\end{array}
\end{equation}
for $a_{i}=(a_{i}^1,\ldots, a_{i}^{n-1})\in \mathcal{L}(\mathcal{N}), ~z\in\mathcal{N}$.\\
We have $\delta^{p+1}\circ\delta^{p}=0$. Thus $(\mathcal{C}^{*}(\mathcal{N},\mathcal{N}),\delta)$ is a cohomology complex for $n$-Lie algebra $\mathcal{N}$. The elements of $Z^{p}:=Ker \delta^{p}$ are $p$-cocycles and   elements of $B^{p}:=Im \delta^{p-1}$ are $p$-coboundaries.
By definition, the $p^{th}$ cohomology group is the quotient space $H^{p}=Z^{p}/B^{p}$.
\end{defn}
\section{Cohomology of $n$-Lie algebras with values in an adjoint module}
The purpose  of this section is to construct cochain complex  $({C}^{*}(\mathcal{N},\mathcal{N}'),\delta)$
that defines a cohomology for $n$-Lie algebras in an adjoint $\mathcal{N}$-module $\mathcal{N}'$.
\begin{defn}
Let $(\mathcal{N},[\cdot ,\ldots,\cdot ])$ and $(\mathcal{N}',[\cdot ,\ldots,\cdot ]')$ be two $n$-Lie algebras and $\phi:\mathcal{N}\rightarrow \mathcal{N}'$ be a $n$-Lie algebra morphism.
Regard $\mathcal{N}'$ as a $\mathcal{N}$-module via the adjoint representation of $\mathcal{N}$ induced by $\phi$. A $\mathcal{N}'$-valued $(m+1)$-cochain is a linear map
$f:\otimes^{m}\mathcal{L}(\mathcal{N})\wedge\mathcal{N}\rightarrow \mathcal{N}'$. We denote by
${C}^{m}(\mathcal{N},\mathcal{N}')$ the set of  $(m+1)$-cochains.
For $m\geq 0$, the coboundary operator $\delta^{m+1}:{C}^{m}(\mathcal{N},\mathcal{N}')\rightarrow
{C}^{m+1}(\mathcal{N},\mathcal{N}')$ is a linear map defined  
by
\begin{equation}\label{cohomology2}
\begin{array}{lllll}
 \delta^{m+1}f(x_{1},\ldots,x_{m},x_{m+1},z)&=&\sum\limits_{1\leq i< j\leq m+1}(-1)^{i}f(x_{1},\ldots,\widehat{x}_{i},\ldots,x_{j-1},[x_{i},x_{j}],\ldots,x_{m+1},z)\\
&+&\sum\limits_{i=1}^{m+1}(-1)^{i}f(x_{1},\ldots,\widehat{x}_{i},\ldots,x_{m+1},ad(x_{i}).z)\\
&+&\sum\limits_{i=1}^{m+1}(-1)^{i+1}L'(x_{i}).f(x_{1},\ldots,\widehat{x}_{i},\ldots,x_{m+1},z)\\
&+&\sum\limits_{i=1}^{n-1}(-1)^{m}[\phi(x_{m+1}^{1}),\ldots,f(x_{1},\ldots,x_{m},x_{m+1}^{i}),
\ldots,\phi(x_{m+1}^{n-1}),\phi(z)]'.
\end{array}
\end{equation}
\end{defn}
\begin{prop}\label{compatibility} We have  
$\delta^{m+2}\circ\delta^{m+1}=0$.
\end{prop}
\begin{proof}
\begin{align}
\nonumber&\delta^{m+2}\circ\delta^{m+1}f(x_{1},\ldots,x_{m+2},z)
=\sum\limits_{i<j}(-1)^{i}\delta^{m+1}f(x_{1},\ldots,\widehat{x}_{i},\ldots
,x_{j-1},[x_{i},x_{j}],\ldots,x_{m+2},z)\\
\nonumber&+\sum\limits_{i=1}^{m+2}(-1)^{i}\delta^{m+1}f(x_{1},\ldots,\widehat{x}_{i},\ldots,x_{m+1},x_{m+2},ad(x_{i}).z)\\
\nonumber&+\sum\limits_{i=1}^{m+2}(-1)^{i+1}L'(x_{i}).\delta^{m+1}f(x_{1},\ldots,\widehat{x}_{i},\ldots,x_{m+2},z)
\\
\nonumber&+\sum\limits_{i=1}^{n-1}(-1)^{m+1}[\phi(x_{m+2}^{1}),\delta^{m+1}f(x_{1},\ldots,x_{m+1},x_{m+2}^{i}),\ldots,
\phi(x_{m+2}^{n-1}),\phi(z)]'
\\
&\tag{a1}\label{a1}
=\sum\limits_{s<t<i<j}(-1)^{i+s}f(x_{1},\ldots,\widehat{x}_{s},\ldots,\widehat{x}_{t},\ldots,[x_{s},x_{t}],\ldots,[x_{i},x_{j}],\ldots,x_{m+2},z)\\
&\tag{a2}\label{a2}
+\sum\limits_{s<i<t<j}(-1)^{i+s}f(x_{1},\ldots,\widehat{x}_{s},\ldots,\widehat{x}_{i},\ldots,[x_{s},x_{t}],\ldots,[x_{i},x_{j}],\ldots,x_{m+2},z)\\
&\tag{a3}\label{a3}
+\sum\limits_{s<i<j<t}(-1)^{i+s}f(x_{1},\ldots,\widehat{x}_{s},\ldots,\widehat{x}_{i},\ldots,[x_{i},x_{j}],\ldots,[x_{s},x_{t}],\ldots,x_{m+2},z)
\\
&\tag{a4}\label{a4}
-\sum\limits_{i<s<t<j}(-1)^{i+s}f(x_{1},\ldots,\widehat{x}_{i},\ldots,\widehat{x}_{s},\ldots,[x_{s},x_{t}],\ldots,[x_{i},x_{j}],\ldots,x_{m+2},z)\\
&\tag{a5}\label{a5}
-\sum\limits_{i<s<j<t}(-1)^{i+s}f(x_{1},\ldots,\widehat{x}_{i},\ldots,\widehat{x}_{s},\ldots,[x_{i},x_{j}],\ldots,[x_{s},x_{t}],\ldots,x_{m+2},z)
\\
&\tag{a6}\label{a6}
-\sum\limits_{i<j<s<t}(-1)^{i+s}f(x_{1},\ldots,\widehat{x}_{i},\ldots,[x_{i},x_{j}],\ldots,\widehat{x}_{s},\ldots,[x_{s},x_{t}],\ldots,x_{m+2},z)\\
&\tag{b1}\label{b1}
+\sum\limits_{k<i<j}(-1)^{i+k}f(x_{1},\ldots,\widehat{x}_{k},\ldots,\widehat{x}_{i},\ldots,[x_{k},[x_{i},x_{j}]],\ldots,x_{m+2},z)\\
&\tag{b2}\label{b2}
+\sum\limits_{i<k<j}(-1)^{i+k}f(x_{1},\ldots,\widehat{x}_{i},\ldots,\widehat{x}_{k},\ldots,[x_{k},[x_{i},x_{j}]],\ldots,x_{m+2},z)\\
&\tag{b3}\label{b3}
+\sum\limits_{i<j<k}(-1)^{i+k}f(x_{1},\ldots,\widehat{x}_{i},\ldots,\widehat{[x_{i},x_{k}]},\ldots,[[x_{i},x_{k}],x_{j}]\ldots,x_{m+1},z)\\
&\tag{c1}\label{c1}
+\sum\limits_{k<i<j}(-1)^{i+k}f(x_{1},\ldots,\widehat{x}_{k},\ldots,\widehat{x}_{i},\ldots,[x_{i},x_{j}],\ldots,x_{m+2}, ad(x_{k}).z)\\
&\tag{c2}\label{c2}
-\sum\limits_{i<k<j}(-1)^{i+k}f(x_{1},\ldots,\widehat{x}_{i},\ldots,\widehat{x}_{k},\ldots,[x_{i},x_{j}],\ldots,x_{m+2},ad(x_{k}).z)\\
&\tag{c3}\label{c3}
-\sum\limits_{i<j<k}(-1)^{i+k}f(x_{1},\ldots,\widehat{x}_{i},\ldots,[x_{i},x_{j}],\ldots,\widehat{x}_{k},\ldots,x_{m+2},ad(x_{k}).z)\\
&\tag{d1}\label{d1}
-\sum\limits_{i<j}(-1)^{i+j}f(x_{1},\ldots,\widehat{x}_{i},\ldots,\widehat{x}_{j},\ldots,x_{m+2},ad([x_{i},x_{j}]).z)\\
&\tag{e1}\label{e1}
+\sum\limits_{k<i<j}(-1)^{i+k+1}L'(x_{k}).f(x_{1},\ldots,\widehat{x}_{k},\ldots,\widehat{x}_{i},\ldots,[x_{i},x_{j}],\ldots,z)\\
&\tag{e2}\label{e2}
-\sum\limits_{i<k<j}(-1)^{i+k+1}L'(x_{k}).f(x_{1},\ldots,\widehat{x}_{i},\ldots,\widehat{x}_{k},\ldots,[x_{i},x_{j}],\ldots,z)\\
&\tag{e3}\label{e3}
-\sum\limits_{i<j<k}(-1)^{i+k+1}L'(x_{k}).f(x_{1},\ldots,\widehat{x}_{i},\ldots,\widehat{x}_{k},\ldots,[x_{i},x_{j}],\ldots,z)\\
&\tag{g1}\label{g1}
-\sum\limits_{i<j}(-1)^{i+j+1}L'([x_{i},x_{j}]).f(x_{1},\ldots,\widehat{x}_{i},\ldots,\widehat{x}_{j},\ldots,x_{m+2},z)
\\
&\tag{h1}\label{h1}
+\sum\limits_{i<j\leq m+1}\sum\limits_{k=1}^{n-1}(-1)^{i+m}[\phi(x_{m+2}^{1}),\ldots,f(x_{1},\ldots,\widehat{x}_{i},\ldots,[x_{i},x_{j}],
\ldots,x_{m+1},\ldots,x_{m+2}^{k}),\ldots,\phi(x_{m+2}^{n-1}),\phi(z)]'\\
&\tag{i1}\label{i1}
+\sum\limits_{k=1}^{m+1}(-1)^{k+m}\sum\limits_{i=1}^{n-1}[[x_{k},x_{m+2}]^1, \ldots, f(x_{1},\ldots,\widehat{x}_{k},\ldots,x_{m+1},[x_{k},x_{m+2}]^i),\ldots,[x_{k},x_{m+2}]^{n-1},z)]\\
&\tag{c4}\label{c4}
+\sum\limits_{s<t<i}(-1)^{s+k}f(x_{1},\ldots,\widehat{x}_{s},\ldots,[x_{s},x_{t}],\ldots,\widehat{x}_{i},\ldots,x_{m+2},ad(x_{i}).z)
\end{align}
\begin{align}
&\tag{c5}\label{c5}
+\sum\limits_{s<i<t}(-1)^{s+k}f(x_{1},\ldots,\widehat{x}_{s},\ldots,\widehat{x}_{i},
\ldots,[x_{s},x_{t}],\ldots,x_{m+2},ad(x_{i}).z)\\
&\tag{c6}\label{c6}
-\sum\limits_{i<s<t}(-1)^{s+k}f(x_{1},\ldots,\widehat{x}_{i},\ldots,\widehat{x}_{s},
\ldots,[x_{s},x_{t}],\ldots,x_{m+2},ad(x_{i}).z)\\
&\tag{d2}\label{d2}
+\sum\limits_{k<i}(-1)^{i+k}f(x_{1},\ldots,\widehat{x}_{k},\ldots,\widehat{x}_{i},\ldots,x_{m+2},
ad(x_{k}).(ad(x_{i}).z))\\
&\tag{d3}\label{d3}
-\sum\limits_{i<k}(-1)^{i+k}f(x_{1},\ldots,\widehat{x}_{i},\ldots,\widehat{x}_{k},\ldots,x_{m+2},
ad(x_{k}).(ad(x_{i}).z))\\
&\tag{p1}\label{p1}
+\sum\limits_{k<i}(-1)^{i+k+1}L'(x_{k}).f(x_{1},\ldots,\widehat{x}_{k},\ldots,\widehat{x_{i}},\ldots,x_{m+2},ad(x_{i}).z)\\
&\tag{p2}\label{p2}
-\sum\limits_{i<k}(-1)^{i+k+1}L'(x_{k}).f(x_{1},\ldots,\widehat{x}_{i},\ldots,\widehat{x_{k}},\ldots,x_{m+2},ad(x_{i}).z)\\
&\tag{i2}\label{i2}
+\sum\limits_{k=1}^{n-1}\sum\limits_{i=1}^{m+1}(-1)^{i+m}
[\phi(x_{m+2}),\ldots,f(x_{1},\ldots,\widehat{x}_{i},\ldots,x_{m+1},x_{m+2}^{k}),\ldots,\phi(x_{m+2}^{n-1}),\phi\circ ad(x_{i}).z]'\\
&\tag{q1}\label{q1}
+\sum\limits_{k=1}^{n-1}[\phi(x_{m+1}^{1}),\ldots,f(x_{1},\ldots,\ldots,x_{m+1}^{k}),\ldots,
\phi\circ ad(x_{m+2}).z]'\\
&\tag{e4}\label{e4}
+\sum\limits_{s<t<i}(-1)^{s+i+1}L'(x_{i}).f(x_{1},\ldots,\widehat{x}_{s},
\ldots,[x_{s},x_{t}],\ldots,\widehat{x}_{i},\ldots,x_{m+2},z)\\
&\tag{e5}\label{e5}
+\sum\limits_{s<i<t}(-1)^{s+i+1}L'(x_{i}).f(x_{1},\ldots,\widehat{x}_{s},
\ldots,\widehat{x}_{i},\ldots,[x_{s},x_{t}],\ldots,x_{m+2},z)\\
&\tag{e6}\label{e6}
+\sum\limits_{i<s<t}(-1)^{s+i+1}L'(x_{i}).f(x_{1},\ldots,\widehat{x}_{i},
\ldots,\widehat{x}_{s},\ldots,[x_{s},x_{t}],\ldots,x_{m+2},z)\\
&\tag{p3}\label{p3}
+\sum\limits_{k<i}(-1)^{i+k}L'(x_{i}).f(x_{1},\ldots,\widehat{x}_{k},\ldots,\widehat{x}_{i},\ldots,x_{m+2},ad(x_{k}).z)\\
&\tag{p4}\label{p4}
-\sum\limits_{i<k}(-1)^{i+k}L'(x_{i}).f(x_{1},\ldots,\widehat{x}_{i},\ldots,\widehat{x}_{k},\ldots,x_{m+2},ad(x_{k}).z)\\
&\tag{g2}\label{g2}
+\sum\limits_{k<i}(-1)^{i+k+1}L'(x_{i}).(L'(x_{k}).f(x_{1},\ldots,\widehat{x}_{k},\ldots,\widehat{x}_{i},\ldots,x_{m+2},z))\\
&\tag{g3}\label{g3}
-\sum\limits_{i<k}(-1)^{i+k+1}L'(x_{i}).(L'(x_{k}).f(x_{1},\ldots,\widehat{x}_{i},\ldots,\widehat{x}_{k},\ldots,x_{m+2},z))
\\
&\tag{i3}\label{i3}
-\sum\limits_{i=1}^{m+1}(-1)^{i+m}L'(x_{i}).(\sum\limits_{k=1}^{n-1}[\phi(x_{m+2}),\ldots,f(x_{1},\ldots,\widehat{x}_{i},\ldots,x_{m+1},
x_{m+2}^{k}),\ldots,\phi(x_{m+2}^{n-1}),\phi(z)]')\\
&\tag{q2}\label{q2}
-L'(x_{m+2}).(\sum\limits_{k=1}^{n-1}[\phi(x_{m+1}^{1}),\ldots,f(x_{1},\ldots,x_{m},x_{m+1}^{k}),\ldots,\phi(x_{m+1}^{n-1}),\phi(z)]')\\
&\tag{h2}\label{h2}
-\sum\limits_{i=1}^{n-1}\sum\limits_{s\leq t\leq m+1}(-1)^{i+m}[\phi(x_{m+2}^{1}),\ldots,f(x_{1},\ldots,\widehat{x}_{s},\ldots,[x_{s},x_{t}]
,\ldots,x_{m+1},x_{m+2}^{i}),\ldots,\phi(x_{m+2}^{n-1}),\phi(z)]'\\
&\tag{i4}\label{i4}
-\sum\limits_{i=1}^{n-1}\sum\limits_{k=1}^{m+1}(-1)^{k+m}
[\phi(x_{m+2}^{1}),\ldots,f(x_{1},\ldots,x_{k},\ldots,x_{m+1},ad(x_{k}).x_{m+2}^{i}),\ldots,\phi(x_{m+2}^{n-1}),\phi(z)]'\\
&\tag{i5}\label{i5}
-\sum\limits_{i=1}^{n-1}\sum\limits_{k=1}^{m+1}(-1)^{k+m}
[\phi(x_{m+2}^{1}),\ldots,L'(x_{i}).f(x_{1},\ldots,x_{i},\ldots,x_{m+2}^{k}),\ldots,\phi(x_{m+2}^{n-1}),\phi(z)]'\\
&\tag{q3}\label{q3}
-\sum\limits_{i=1}^{n-1}\sum\limits_{k=1}^{n-1}[\phi(x_{m+2}^{1}),\ldots,[\phi(x_{m+1}^{1}),\ldots,f(x_{1},\ldots,x_{m},x_{m+1}^{k})
,\ldots,\phi(x_{m+1}^{n-1}),\phi(x_{m+2}^{i})]',\ldots,\phi(x_{m+2}^{n-1}),\phi(z)]'
\end{align}
We will show that the sum of terms named by the same letter vanish. Indeed, thanks to identity \eqref{LeibnizIdentity} in  $(\mathcal{L}(\mathcal{N}),[\cdot,\cdot])$, which  is a Leibniz algebra), we get
\begin{eqnarray*}
\eqref{b1}+\eqref{b2}+\eqref{b3}&=&\sum\limits_{i<k<j}(-1)^{i+k}f(x_{1},\ldots,\widehat{x}_{i},\ldots,\widehat{x}_{k},\ldots,[x_{k},[x_{i},x_{j}]]
 ,\ldots,x_{m+1} ,z)\\
 &-&\sum\limits_{i<k<j}(-1)^{i+k}f(x_{1},\ldots,\widehat{x}_{i},\ldots,\widehat{x}_{k},\ldots,[x_{i},[x_{k},x_{j}]],\ldots,x_{m+1},z)\\
 &-&\sum\limits_{i<k<j}(-1)^{i+k}f(x_{1},\ldots,\widehat{x}_{i},\ldots,\widehat{[x_{i},x_{k}]},\ldots,[[x_{i},x_{k}],x_{j}],\ldots,x_{p+1},z)\\
 &=&0.
\end{eqnarray*}
Thanks to the property (\ref{lem2}), we get
$\eqref{d1}+\eqref{d2}+\eqref{d3}=0$. Now, we have

\begin{equation}\label{ni}
\eqref{i1}+\eqref{i2}+\eqref{i3}+\eqref{i4}+\eqref{i5}=0.
\end{equation}
Indeed: First, we can see that  formulas \eqref{i1} and \eqref{i3} can be expressed as follows:
\begin{align}
\tag{i1a}\label{i1a}
\eqref{i1}&=\sum\limits_{k=1}^{m+1}(-1)^{k+m}\sum\limits_{i=1}^{n-1}\sum\limits_{j<i}
[\phi(x_{m+2}^{1}),\ldots,f(x_{1},\ldots,\widehat{x}_{k},\ldots,x_{m+1},x_{m+2}^{j}),\ldots,\phi \circ ad(x_{k}).x_{m+2}^{i},\ldots,
\phi(x_{m+2}^{n-1}),\phi(z)]'\\
&\tag{i1b}\label{i1b}
+\sum\limits_{k=1}^{m+1}(-1)^{k+m}\sum\limits_{i=1}^{n-1}\sum\limits_{j>i}
[\phi(x_{m+2}^{1},\ldots,\phi \circ ad(x_{k}).x_{m+2}^{i},\ldots,f(x_{1},\ldots,\widehat{x}_{k},\ldots,x_{m+1},x_{m+2}^{j}),\ldots,
\phi(x_{m+2}^{n-1}),\phi(z)]'\\
&\tag{i1c}\label{i1c}
-\sum\limits_{k=1}^{m+1}(-1)^{k+m}\sum\limits_{i=1}^{n-1}[\phi(x_{m+2}^{1}),\ldots,f(x_{1},\ldots,\widehat{x}_{k},\ldots,x_{m+2},
ad(x_{k}).x_{m+2}^{i}),\ldots,\phi(x_{m+2}^{n-1}),\phi(z)]'.
\end{align}
\begin{align}
\tag{i3a}\label{i3a}
\eqref{i3}&=-\sum\limits_{i=1}^{m+1}(-1)^{i+m}\sum\limits_{j=1}^{n-1}\sum_{l<j}[\phi(x_{m+2}^{1}),\ldots,\phi\circ ad(x_{i})\cdot x_{m+2}^{l}
,\ldots,f(x_{1},\ldots,\widehat{x}_{i},\ldots,x_{m+1},x_{m+2}^{j}),\ldots,\phi(x_{m+2}^{n-1}),\phi(z)]'\\
&\tag{i3b}\label{i3b}
-\sum\limits_{i=1}^{m+1}(-1)^{i+m}\sum\limits_{j=1}^{n-1}\sum_{l>j}[\phi(x_{m+2}^{1}),
,\ldots,f(x_{1},\ldots,\widehat{x}_{i},\ldots,x_{m+1},x_{m+2}^{j}),\ldots,\phi\circ ad(x_{i})\cdot x_{m+2}^{l},
\ldots,\phi(x_{m+2}^{n-1}),\phi(z)]'\\
&\tag{i3c}\label{i3c}
-\sum\limits_{i=1}^{m+1}(-1)^{i+m}\sum\limits_{j=1}^{n-1}[\phi(x_{m+2}^{1}),
,\ldots,L'(x_{i})\cdot f(x_{1},\ldots,\widehat{x}_{i},\ldots,x_{m+1},x_{m+2}^{j}),\ldots,
\ldots,\phi(x_{m+2}^{n-1}),\phi(z)]'\\
&\tag{i3d}\label{i3d}
-\sum\limits_{i=1}^{m+1}(-1)^{i+m}\sum\limits_{j=1}^{n-1}[\phi(x_{m+2}^{1}),
,\ldots,f(x_{1},\ldots,\widehat{x}_{i},\ldots,x_{m+1},x_{m+2}^{j}),\ldots,
\ldots,\phi(x_{m+2}^{n-1}),\phi\circ ad(x_{i})\cdot z]'.
\end{align}
Second, we  check that
\[
\eqref{i4}+\eqref{i1c}=0,~ \eqref{i2}+\eqref{i3d}=0, ~\eqref{i1a}+\eqref{i3b}=0,~ \eqref{i1b}+\eqref{i3a}=0~
\hbox{ and } \eqref{i3c}+\eqref{i5}=0.
\]
Then, we get formula (\ref{ni}).

Next, we have
 \begin{equation}\label{ni1}
 \eqref{q1}+\eqref{q2}+\eqref{q3}=0.
  \end{equation}
Indeed: First, we can see that  formulas \eqref{q2} and \eqref{q3} can be expressed as follows:

\begin{align}
\tag{$q_{2}1$}\label{$q_{2}1$}
\eqref{q2}&=-\sum\limits_{k=1}^{n-1}\sum\limits_{i\neq k}[\phi(x_{m+1}^{1}),\ldots,\phi\circ ad(x_{m+2})\cdot x_{m+1}^{i},\ldots,
f(x_{1},\ldots,x_{m},x_{m+1}^{k}),\ldots,\phi(x_{m+1}),\phi(z)]'\\
&\tag{$q_{2}2$}\label{$q_{2}2$}
-\sum\limits_{k=1}^{n-1}[\phi(x_{m+1}^{1}),\ldots,L'(x_{m+2})\cdot f(x_{1},\ldots,x_{m},x_{m+1}^{k}),\ldots,\phi(x_{m+1}^{n-1}),\phi(z)]'\\
&\tag{$q_{2}3$}\label{$q_{2}3$}
-\sum\limits_{k=1}^{n-1}[\phi(x_{m+1}^{1}),\ldots,f(x_{1},\ldots,x_{m},x_{m+1}^{k}),\ldots,\phi(x_{m+1}^{n-1}),
\phi\circ ad(x_{m+2})\cdot z]'
\end{align}

\begin{align}
\tag{$q_{3}1$}\label{$q_{3}1$}
 \eqref{q3}&=-\sum\limits_{k=1}^{n-1}[\phi(x_{m+1}^{1}),\ldots,f(x_{1},\ldots,x_{m},x_{m+1}^{k}),\ldots,\phi(x_{m+1}^{n-1}),
[\phi(x_{m+2}^{1}),\ldots,\phi(x_{m+2}^{n-1}),\phi(z)]']'\\
&\tag{$q_{3}2$}\label{$q_{3}2$}
+ \sum\limits_{k=1}^{n-1}[\phi(x_{m+2}^{1}),\ldots,\phi(x_{m+2}^{n-1}),[\phi(x_{m+1}^{1}),\ldots,f(x_{1},\ldots,x_{m},x_{m+1}^{k}),
\ldots,\phi(x_{m+1}^{n-1}),\phi(z)]']'
\end{align}
Second, we  check that
$(\eqref{$q_{2}3$}+\eqref{q1}=0$ and $(\eqref{$q_{3}1$}+\eqref{$q_{2}1$}+\eqref{$q_{2}2$}+\eqref{$q_{3}2$}=0$.
Then, we get formula (\ref{ni1}). Finally, we can see that the other case are equal to $0$ by direct calculation.
\end{proof}

\begin{defn}
The space of $(n+1)$-cocycles is defined by
$$Z^{n+1}(\mathcal{N}, \mathcal{N}')=\{\varphi\in{\mathcal{C}}^{n}
(\mathcal{N}, \mathcal{N}'):\delta^{n+1}\varphi=0\},$$
and the space of $(n+1)$-coboundaries is defined by
$$B^{n+1}(\mathcal{N}, \mathcal{N}')=\{\psi=\delta^{n}\varphi:\varphi
\in{\mathcal{C}}^{n-1}(\mathcal{N},\mathcal{N}')\}.$$
One has $B^{n+1}(\mathcal{N}, \mathcal{N}')\subset Z^{n+1}(\mathcal{N}, \mathcal{N}')$.
Then,
we call the $(n+1)^{th}$ cohomology group of the $n$-Lie algebra $\mathcal{N}$ with coefficients in $\mathcal{N}'$,
the quotient
$$H^{n+1}(\mathcal{N}, \mathcal{N}')=\frac{Z^{n+1}(\mathcal{N}, \mathcal{N}')}{B^{n+1}(\mathcal{N}, \mathcal{N}')}.$$
\end{defn}
\section{Cohomology complex of $n$-Lie algebra morphisms}
The original cohomology theory associated to deformation of Lie algebra morphisms was developed
by Fr\'{e}gier in \cite{f}.
The aim of  this section is to provide the main result of this paper, that is a generalization of 
this theory to $n$-Lie algebra morphisms.

Let $\phi:\mathcal{N}\rightarrow\mathcal{N}'$ be a $n$-Lie algebra morphism.
Regard $\mathcal{N}'$ as a representation of $\mathcal{N}$ via $\phi$ wherever
appropriate.
We define the module of $(m+1)$-cochains of the morphism $\phi$ to be
\begin{equation}\label{do}
\mathcal{C}^{m}(\phi,\phi)=\mathcal{C}^{m}(\mathcal{N},\mathcal{N})\otimes \mathcal{C}^{m}(\mathcal{N}',\mathcal{N}')
\otimes {\mathcal{C}}^{m-1}(\mathcal{N},\mathcal{N}').
\end{equation}
The coboundary operator $\delta^{m+1}:\mathcal{C}^{m}(\phi,\phi)\rightarrow \mathcal{C}^{m+1}(\phi,\phi)$ is
defined by
$$\delta^{m+1}(\varphi_{1}, \varphi_{2}, \varphi_{3})=(\delta^{m+1}\varphi_{1},
 \delta^{m+1}\varphi_{2}, \delta^{m}\varphi_{3}+(-1)^{m}(\phi\circ\varphi_{1}-
\varphi_{2}\circ(\bar{\phi}^{\otimes m}\wedge\phi))),$$
where $\delta^{m+1}\varphi_{1}$ and $\delta^{m+1}\varphi_{2}$ are defined by \eqref{cohomology},
  $\delta^{m}\varphi_{3}$
is defined by \eqref{cohomology2} and $\bar{\phi}^{\otimes m}\wedge\phi:\otimes^{m}\mathcal{L}(\mathcal{N})\wedge\mathcal{N}\rightarrow\otimes^{m}\mathcal{L}(\mathcal{N}')\wedge\mathcal{N}'$
is defined by:
$$(\bar{\phi}^{\otimes m}\wedge\phi)(x_{1},\ldots,x_{m},z)=(\bar{\phi}(x_{1}),\ldots,\bar{\phi}(x_{m}),\phi(z))
\quad \hbox{for}~x_i\in\mathcal{L}(\mathcal{N})~\hbox{and}~z\in\mathcal{N}.
$$
\begin{prop} We have $\delta^{m+2}\circ\delta^{m+1}=0$. Hence
$(C^{*}(\phi,\phi),\delta)$ is a cochain complex.
\end{prop}
\begin{proof}
The most right  component of $(\delta^{m+2}\circ\delta^{m+1})
(\varphi_{1},\varphi_{1},\varphi_{3})$ is 
$$(-1)^{m}\delta^{m+1}(
\phi\circ\varphi_{1}-\varphi_{2}\circ(\bar{\phi}^{\otimes m}\wedge\phi))+(-1)^{m+1}
(\phi\circ\delta^{m+1}(\varphi_{1})-\delta^{m+1}(\varphi_{2})\circ(\bar{\phi}^{\otimes (m+1)}\wedge\phi)).
$$
Thus, to finish the proof, one checks that
$\delta^{m+1}(\phi\circ\varphi_{1})=\phi\circ\delta^{m+1}(\varphi_{1})$ and
$\delta^{m+1}(\varphi_{2}\circ(\bar{\phi}^{\otimes m}\wedge\phi))=\delta^{m+1}(\varphi_{2})\circ(\bar{\phi}^{\otimes (m+1)}\wedge\phi)$.
Indeed:
\[
\begin{array}{ll}
 \delta^{m+1}(\phi\circ\varphi_{1})(x_{1},\ldots,x_{m+1},z)=&\sum\limits_{1\leq i< j\leq m+1}(-1)^{i}(\phi\circ\varphi_{1})(x_{1},\ldots,\widehat{x}_{i},\ldots,x_{j-1},[x_{i},x_{j}],\ldots,x_{m+1},z)\\
&+~\sum\limits_{i=1}^{m+1}(-1)^{i}(\phi\circ\varphi_{1})(x_{1},\ldots,\widehat{x}_{i},\ldots,x_{m+1},ad(x_{i}).z)
\\ & +~\sum\limits_{i=1}^{m+1}(-1)^{i+1}L'(x_{i}).(\phi\circ\varphi_{1})(x_{1},\ldots,\widehat{x}_{i},\ldots,x_{m+1},z)\\
&+~\sum\limits_{i=1}^{n-1}(-1)^{m}[\phi(x_{m+1}^{1}),\ldots,(\phi\circ\varphi_{1})(x_{1},\ldots,x_{m},x_{m+1}^{i}),
\ldots,\phi(x_{m+1}^{n-1}),\phi(z)]'\\
&=\phi\bigg(\sum\limits_{1\leq i< j\leq m+1}(-1)^{i}\varphi_{1}(x_{1},\ldots,\widehat{x}_{i},\ldots,x_{j-1},[x_{i},x_{j}],\ldots,x_{m+1},z)\\
&+~\sum\limits_{i=1}^{m+1}(-1)^{i}\varphi_{1}(x_{1},\ldots,\widehat{x}_{i},\ldots,x_{m+1},ad(x_{i}).z)\\
&+~\sum\limits_{i=1}^{m+1}(-1)^{i+1}ad(x_{i}).\varphi_{1}(x_{1},\ldots,\widehat{x}_{i},\ldots,x_{m+1},z)\\
&+~\sum\limits_{i=1}^{n-1}(-1)^{m}[x_{m+1}^{1},\ldots,\varphi_{1}(x_{1},\ldots,x_{m},x_{m+1}^{i}),\ldots,x_{m+1}^{n-1},z]\bigg)\\
&=\phi\circ\delta^{m+1}(\varphi_{1})(x_{1},\ldots,x_{m+1},z).\\[10pt]
\delta^{m+1}(\varphi_{2})\circ(\bar{\phi}^{\otimes (m+1)}\wedge\phi)(x_{1},\ldots,x_{m+1},z)&=\sum\limits_{1\leq i< j\leq m+1 }(-1)^{i}\varphi_{2}(\bar{\phi}(x_{1}),\ldots,\widehat{\bar{\phi}({x}_{i})},\ldots,\bar{\phi}(x_{j-1}),[\bar{\phi}(x_{i}),\bar{\phi}(x_{j})]',\ldots,
\bar{\phi}(x_{m+1}),\phi(z))\\
&+~\sum\limits_{i=1}^{m+1}(-1)^{i}\varphi_{2}(\bar{\phi}(x_{0}),\ldots,\widehat{\bar{\phi}({x}_{i})},\ldots,\bar{\phi}(x_{m+1}),L'(x_{i}).\phi(z))\\
&+~\sum\limits_{i=1}^{m+1}(-1)^{i+1}L'(x_{i}).\varphi_{2}(\bar{\phi}(x_{1}),\ldots,\widehat{\bar{\phi}({x}_{i})},\ldots,\bar{\phi}(x_{m+1}),\phi(z))\\
 &+~\sum\limits_{i=1}^{n-1}(-1)^{m}[\phi(x_{m+1}^{1}),\ldots,\varphi_{2}(\bar{\phi}(x_{1}),\ldots,\bar{\phi}(x_{m}),\phi(x_{m+1}^{i})),
\ldots,\phi(x_{m+1}^{n-1}),\phi(z)]'\\
&=\delta^{m+1}(\varphi_{2}\circ(\bar{\phi}^{\otimes m}\wedge\phi))(x_{1},\ldots,x_{m+1},z).
\end{array}
\]
\end{proof}

\begin{defn}
The corresponding cohomology modules of the cochain complex $(C^{*}(\phi,\phi),\delta^{n})$
are denoted by
$$H^{n+1}(\phi,\phi):=H^{n+1}(C^{*}(\phi,\phi),\delta).$$
\end{defn}
\begin{prop} 
If $H^{n+1}(\mathcal{N},\mathcal{N}),~H^{n+1}(\mathcal{N}',\mathcal{N}')$,
and $H^{n}(\mathcal{N},\mathcal{N}')$ are all trivial then so is $H^{n+1}(\phi,\phi)$.
\end{prop}
\begin{proof}
The proof is similar to that of Proposition 3.3 in \cite{d}.
\end{proof}
\section{Deformations of $n$-Lie algebra morphisms}
In  this section, we aim  to study one parameter formal deformation of $n$-Lie algebra morphisms. Deformation theory using using formal power series was introduced first for associative algebras  by Gerstenhaber  \cite{Gerstenhaber} and then extended to Lie algebras by Nijenhuis and Richardson  \cite{NR}. 
Deformations of $n$-Lie algebras has been discussed in terms of
Chevalley-Eilenberg cohomology in many articles, see \cite{A} for a review.
Recall that the main idea is to change the
scalar field $\mathbb{K}$ to a formal power series ring $\mathbb{K}[\![t]\!]$, in one variable $t$, and the main results
provide cohomological interpretations.

Let $\mathbb{K}[\![t]\!]$ be the power series in one variable $t$ and coefficients in $\mathbb{K}$
and $\mathcal{N}[\![t]\!]$ be the set of formal series whose coefficients are elements of the vector
space $\mathcal{N}$, ($\mathcal{N}[\![t]\!]$ is obtained by extending the coefficients domain of $\mathcal{N}$
from $\mathbb{K}$ to $\mathbb{K}[\![t]\!]$). Given a $\mathbb{K}$-$n$-linear map
$\varphi:\mathcal{N}\times\ldots\mathcal{N}\rightarrow\mathcal{N}$, it admits naturally an extension to a
$\mathbb{K}[\![t]\!]$-$n$-linear map $\varphi:\mathcal{N}[\![t]\!]\times\ldots\times\mathcal{N}[\![t]\!]
\rightarrow \mathcal{N}[\![t]\!]$, that is, if $x_{i}=\sum\limits_{j\geq0}a_{i}^{j}t_{i}^j, 1\leq i \leq n$
then $\varphi(x_{1},\ldots,x_{n})=\sum\limits_{j_{1},\ldots,j_{n}\geq0}t^{j_{1}+\ldots+j_{n}}\varphi
(a_{1}^{j_{1}},\ldots,a_{n}^{j^{n}})$.

\begin{defn}
Let $(\mathcal{N},[\cdot,\ldots,\cdot])$ be a $n$-Lie algebra. A one-parameter formal deformation
of the $n$-Lie algebra $\mathcal{N}$ is given by a $\mathbb{K}[\![t]\!]$-$n$-linear map
$$[\cdot,\ldots,\cdot]_{t}:\mathcal{N}[\![t]\!]\times\ldots\times\mathcal{N}[\![t]\!]\rightarrow\mathcal{N}[\![t]\!]$$
of the form $[\cdot,\ldots,\cdot]_{t}=\sum\limits_{i\geq0}t^{i}[\cdot,\ldots,\cdot]_{i}$ where each $[\cdot,\ldots,\cdot]_{i}$
is a skew-symmetric $\mathbb{K}$-$n$-linear map $[\cdot,\ldots,\cdot]_{i}:\mathcal{N}\times\ldots \times\mathcal{N}
\rightarrow\mathcal{N}$
(extended to a $\mathbb{K}[\![t]\!]$-$n$-linear map), and $[\cdot,\ldots,\cdot]_{0}=[\cdot,\ldots,\cdot]$ such that for
$(x_{i})_{1\leq i \leq 2n-1}$
\begin{equation}\label{def}
[x_{1},\ldots,x_{n-1},[x_{n},\ldots,x_{2n-1}]_{t}]_{t}=\sum\limits_{i=n}^{2n-1}
[x_{n},\ldots,x_{i-1},[x_{1},\ldots,x_{n-1},x_{i}]_{t},x_{i+1},\ldots,x_{2n-1}]_{t}
\end{equation}
\end{defn}
The deformation is said to be of order $k$ if $[\cdot,\ldots,\cdot]_{t}=\sum\limits_{i=0}^{k}t^{i}[\cdot,\ldots,\cdot]_{i}$
and infinitesimal if $k=1$.
\begin{defn}
Let $\phi:\mathcal{N}\rightarrow \mathcal{N}'$ be a $n$-Lie algebra morphism.
Define a deformation of $\phi$ to be a triple \\$\Theta_{t}=([\cdot,\ldots,\cdot]_{\mathcal{N},t};[\cdot,\ldots,\cdot]_{\mathcal{N}',t};\phi_{t})$
in which :
\begin{itemize}
    \item $[\cdot,\ldots,\cdot]_{\mathcal{N},t}=\sum\limits_{i\geq0}t^{i}[\cdot,\ldots,\cdot]_{\mathcal{N},i}$ is a deformation of $\mathcal{N}$
    \item $[\cdot,\ldots,\cdot]_{\mathcal{N}',t}=\sum\limits_{i\geq0}t^{i}[\cdot,\ldots,\cdot]_{\mathcal{N}',i}$ is a deformation of $\mathcal{N}'$
    \item $\phi_{t}:\mathcal{N}[\![t]\!]\rightarrow \mathcal{N}'[\![t]\!]$ is
    a $n$-Lie algebra morphism of the form $\phi_{t}=\sum\limits_{n\geq0}\phi_{n}t^{n}$
    where each $\phi_{n} :\mathcal{N}\rightarrow \mathcal{N}'$ is a $\mathbb{K}$-linear map and
    $\phi_{0}=\phi$, such that $\phi_{t}$ satisfies the following equation
    $$\phi_{t}([x_{1},\ldots,x_{n}]_{\mathcal{N},t})=[\phi_{t}(x_{1}),\ldots,\phi_{t}(x_{n})]_{\mathcal{N}',t}.$$
\end{itemize}
\end{defn}
\begin{prop}
\label{nissma}
The linear coefficient, $\theta_{1} = ([., .]_{1}, [., .]_1', \phi_1)$, which is called the infinitesimal
of the deformation $\Theta_{t}$ of $\phi$
is a $2$-cocycle in $C^{2}(\phi,\phi)$.
\end{prop}
\begin{proof}
Let $\phi: \mathcal{N}\rightarrow\mathcal{N}'$ be a $n$-Lie algebra morphism, we have the following deformation equation
\begin{eqnarray*}
\phi_{t}([x_{1},\ldots,x_{n}]_{t})=[\phi_{t}(x_{1}),\ldots,\phi_{t}(x_{n})]'_{t}.
\end{eqnarray*}
Expanding this product in a power series in $t$, we obtain
\begin{eqnarray*}
\sum\limits_{i+j=s}\phi_{i}([x_{1},\ldots,x_{n}]_{j})=
\sum\limits_{i_{1}+\ldots+i_{n}+j=s}[\phi_{i_{1}}(x_{1}),\ldots,\phi_{i_{n}}(x_{n})]'_{j}.
\end{eqnarray*}
For $s=1$, we get
\[
\phi([x_{1},\ldots,x_{n}]_{1})-[\phi(x_{1}),\ldots,\phi(x_{n})]_{1}'
-\sum\limits_{i=1}^{n}[\phi(x_{1}),\ldots,\phi_{1}(x_{i}),\ldots,\phi(x_{n})]'
+\phi_{1}([x_{1},\ldots,x_{n}])=0.
\]
This is equivalent to the $2$-cochain
$\phi([x_{1},\ldots,x_{n}]_{1})-[\phi(x_{1}),\ldots,\phi(x_{n})]_{1}'
-\delta^{1}\phi_{1}$ is equal to $0$.
\end{proof}
\subsubsection{Equivalent deformations}
\begin{defn}
Let $(\mathcal{N},[\cdot ,\ldots,\cdot ])$ be a $n$-Lie algebra. Given two deformations
$\mathcal{N}_{t}=(\mathcal{N}[\![t]\!],[\cdot ,\ldots,\cdot ]_{t})$ and $\mathcal{N}'_{t}=(\mathcal{N}[\![t]\!],[\cdot ,\ldots,\cdot ]'_{t})$
of $\mathcal{N}$. 
We say that $\mathcal{N}_{t}$ and $\mathcal{N}'_{t}$ are equivalent if there exists a formal automorphism
$\psi_{t}:\mathcal{N}[\![t]\!]\rightarrow \mathcal{N}[\![t]\!]$ that may be written in the form $\psi_{t}=\sum\limits_{i\geq 0}t^{i}\psi_{i}$, where
$\psi_{i}\in End (\mathcal{N})$ and $\psi_{0}=Id$ such that
$$\psi_{t}([x_{1},\ldots,x_{n}]_{t})=[\psi_{t}(x_{1}),\ldots,\psi_{t}(x_{n})]'_{t}.$$
\end{defn}
\noindent A deformation $\mathcal{N}_{t}$ of $\mathcal{N}$ is said to be trivial if $\mathcal{N}_{t}$ is equivalent to
$\mathcal{N}$, viewed as a $n$-ary algebra on $\mathcal{N}[\![t]\!]$.\\
Let $(\mathcal{N},[\cdot ,\ldots,\cdot ])$ be a $n$-Lie algebra and $[\cdot ,\ldots,\cdot ]_{1}\in Z^{2}(\mathcal{N},\mathcal{N})$.
The $2$-cocycle $[\cdot ,\ldots,\cdot ]_{1}$ is said to be integrable if there exists a family
$([\cdot ,\ldots,\cdot ]_{i})_{i\geq0}$ such that $[\cdot ,\ldots,\cdot ]_{t}=\sum\limits_{i\geq0}t^{i}[\cdot ,\ldots,\cdot ]_{i}$
defines a formal deformation $\mathcal{N}_{t}=(\mathcal{N}[\![t]\!],[\cdot ,\ldots,\cdot ]_{t})$ of $\mathcal{N}$.
\begin{prop}\label{prop4.5}
If $\mathcal{N}_{t}$ and $\mathcal{N}'_{t}$ are equivalent deformations of $\mathcal{N}$
given by the automorphism $\psi_{\mathcal{N},t}:\mathcal{N}[\![t]\!]\rightarrow\mathcal{N}[\![t]\!]$,
the infinitesimals of $[\cdot ,\ldots,\cdot ]_{\mathcal{N},t}$ and $[\cdot ,\ldots,\cdot ]_{\mathcal{N}',t}$ belong to the same
cohomology class.
\end{prop}
\begin{proof}
The proof is straightforward  and 
similar to that 
of the case $n=2$.
\end{proof}
\begin{defn}\label{DefMorphEquiv}
Let $\Theta_{t}=([\cdot ,\ldots,\cdot ]_{\mathcal{N},t},[\cdot ,\ldots,\cdot ]_{\mathcal{L},t},\phi_{t})$ and $\widetilde{\Theta}=
([\cdot ,\ldots,\cdot ]'_{\mathcal{N},t},[\cdot ,\ldots,\cdot ]'_{\mathcal{L},t},\widetilde{\phi}_{t})$ be two deformations of
a $n$-Lie algebra morphism $\phi:\mathcal{N}\rightarrow\mathcal{L}$.
A formal automorphism $:\Theta_{t}\rightarrow\widetilde{\Theta}_{t}$ is
a pair $(\psi_{\mathcal{N},t},\psi_{\mathcal{L},t})$, where $\psi_{\mathcal{N},t}:\mathcal{N}[\![t]\!]\rightarrow \mathcal{N}[\![t]\!]$
and $\psi_{\mathcal{L},t}:\mathcal{\mathcal{L}}[\![t]\!]\rightarrow \mathcal{L}[\![t]\!]$ are formal automorphisms, such that
$\widetilde{\phi}_{t}=\psi_{\mathcal{L},t}\circ\phi_{t}\circ\psi^{-1}_{\mathcal{N},t}$.
Two deformations $\Theta_{t}$ and $\widetilde{\Theta}_{t}$ are equivalent
if and only if there exists a formal automorphism $\Theta_{t}\rightarrow\widetilde{\Theta}_{t}$.
\end{defn}
Given a deformation $\Theta_{t}$ and a pair of power series
$\psi_{t}=(\psi_{\mathcal{N},t}=\sum\limits_{n}\psi_{\mathcal{N},n}t^{n},\psi_{\mathcal{L},t}=\sum\limits_{n}\psi_{\mathcal{L},n}t^{n})$,
one can define a deformation $\widetilde{\Theta}_{t}$ which
is automatically equivalent to $\Theta_{t}$.
\begin{thm} Let $(\mathcal{N},[\cdot ,\ldots,\cdot ]_{\mathcal{N}})$ and $(\mathcal{N}',[\cdot ,\ldots,\cdot ]_{\mathcal{N}'})$
be two $n$-Lie algebras. Let $\Theta_{t}=([\cdot ,\ldots,\cdot ]_{\mathcal{N},t},[\cdot ,\ldots,\cdot ]_{\mathcal{N}',t},\phi_{t})$ be a
deformation of a $n$-Lie algebra morphism $\phi: \mathcal{N}\rightarrow\mathcal{N}'$. Then
\begin{enumerate}
  \item The infinitesimal of a deformation $\Theta_{t}$ of $\phi$ is a $2$-cocycle in
$C^{2}(\phi,\phi)$ whose cohomology class is determined by the equivalence class of $\Theta_{t}$.
  \item There exists an equivalent deformation $\widetilde{\Theta}_{t}=([\cdot ,\ldots,\cdot ]_{\mathcal{N},t},
[\cdot ,\ldots,\cdot ]_{\mathcal{N}',t},\widetilde{\phi}_t)$ such that
$\widetilde{\theta}_{1}\in Z^{2}(\phi,\phi)$ and $\widetilde{\theta}_{1}
\not\in B^{2}(\phi,\phi)$. Hence, if $H^{2}(\phi,\phi)=0$ then every formal deformation is
equivalent to a trivial deformation.
\end{enumerate}
\end{thm}
\begin{proof}
\begin{enumerate}
  \item By Proposition \ref{nissma}, $\theta_{1}$ is a $2$-cocycle. Now, if $\psi_{t}$ is a formal automorphism, then the $2$-cocycle $\theta_{1}$
and $\widetilde{\theta}_{1}$ differ by a $2$-coboundary. Write $\psi_{t}=(\psi_{\mathcal{N},t},\psi_{\mathcal{L},t})$
and $\widetilde{\theta}_{t}=([\cdot,\ldots,\cdot]'_{\mathcal{N},t},[\cdot,\ldots,\cdot]'_{\mathcal{L},t},\widetilde{\phi}_{t})$.
In view of Proposition \ref{prop4.5}, we have $\delta^{1}\psi_{*,1}=[\cdot,\ldots,\cdot]_{*,1}-[\cdot,\ldots,\cdot]'_{*,1}\in C^{2}(*,*)$,
for $*\in\{\mathcal{N},\mathcal{L}\}$. Moreover,
$\phi_{1}-\widetilde{\phi}_{1}=\phi\circ\psi_{\mathcal{N},1}-\psi_{\mathcal{L},1}\circ\phi$, then
$\theta_{1}-\widetilde{\theta}_{1}=\delta^{1}\alpha$, with $\alpha=(\psi_{\mathcal{N},1},\psi_{\mathcal{L},1},0)$.
  \item Define a pair of power series $\psi_{t}=(\psi_{\mathcal{N},t},\psi_{\mathcal{L},t})$.  According to Definition \ref{DefMorphEquiv},
  we  define equivalent deformation $\widetilde{\Theta}=
([\cdot,\ldots,\cdot]'_{\mathcal{N},t},[\cdot ,\ldots,\cdot ]'_{\mathcal{L},t},\widetilde{\phi}_{t})
  =\sum\limits_{n}\widetilde{\theta}_{n}t^{n}$.
We have $[\cdot,\ldots,\cdot]_{*,1}\in Z^{2}(*,*)$ and also $[\cdot,\ldots,\cdot]_{*,1}-[\cdot,\ldots,\cdot]'_{*,1}
\in Z^{2}(*,*)$ for $*\in\{\mathcal{N},\mathcal{L}\}$. Moreover  $\phi_{1}\in Z^{1}(\mathcal{N},\mathcal{L})$
leads to $\phi_{1}-\widetilde{\phi}_{1}\in Z^{1}(\mathcal{N},\mathcal{L})$. If $\widetilde{\theta}_{1}
\in B^{2}(\phi,\phi)$ then so
$\theta_{1}-\widetilde{\theta}_{1}=\delta^{1}\varphi$ for $\varphi\in C^{1}(\phi,\phi)$.
\end{enumerate}
\end{proof}
\subsection{Obstructions}
Let $(\mathcal{N},[\cdot ,\ldots,\cdot ])$ and $(\mathcal{N}',[\cdot ,\ldots,\cdot ]')$ be  two $n$-Lie algebras
and let $\phi: \mathcal{N}\rightarrow\mathcal{N}'$ be a morphism of $n$-Lie algebras. A deformation of order $N$ of $\phi$ is  a triple,
$\Theta_{t}=([\cdot,\ldots,\cdot]_{t};[\cdot,\ldots,\cdot]'_{t};\phi_{t})$
satisfying $\phi_{t}([x_{1},\ldots,x_{n}]_{t})=[\phi_{t}(x_{1}),\ldots,\phi_{t}(x_{n})]'_{t}$
or equivalently
$$\sum\limits_{i+j=N}\phi_{i}\circ[x_{1},\ldots,x_{n}]_{j}=\sum\limits_{i_{1}+\cdots+i_{n}+j=N}
[\phi_{i_{1}}(x_{1}),\ldots,\phi_{i_{n}}(x_{n})]'_{j}.$$
Given a deformation $\Theta_{t}$ of order $N$, it is said to extend to order $N+1$
if and only if there exists a $2$-cochain $\theta_{N+1}=([\cdot ,\ldots,\cdot ]_{N+1},[\cdot ,\ldots,\cdot ]_{N+1}',\phi_{N+1})
\in C^{2}(\phi,\phi)$ such that $\overline{\Theta}_{t}=\Theta_{t}+t^{N+1}\theta_{N+1}$
is a deformation of order $N+1$.
$\overline{\Theta}_{t}$ is called an order $N+1$ extension of $\Theta_{t}$.
Now, for $x_i=(x_i^1,\ldots,x_{i}^{n-1})\in\mathcal{L}(\mathcal{N})$ and $z\in\mathcal{N}$, 
the $(N+1)$-equation of \eqref{def} $\hbox{can be written as}$
\begin{equation}
\label{nizar}
\begin{array}{lllll}
\delta^{3}\left([\cdot ,\ldots,\cdot ]_{N+1}\right)(x_1,x_2,z)&=&-\sum\limits_{\begin{subarray}{l}k+l=N+1\\~~k,l>0
\end{subarray}}[x_{1}^{1},\ldots,x_{1}^{n-1},[x_{2}^{1},\ldots,x_{2}^{n-1},z]_{k}]_{l}
+\sum\limits_{\begin{subarray}{l}k+l=N+1\\~~k,l>0
\end{subarray}}[x_{2}^{1},\ldots,x_{2}^{n-1},[x_{1}^{1},\ldots,x_{1}^{n-1},z]_{k}]_{l}
\\[1pt]
& &+
\sum\limits_{\begin{subarray}{l}k+l=N+1\\~~k,l>0
\end{subarray}}\sum\limits_{i=1}^{n-1}[x_{2}^{1},\ldots,x_{2}^{i-1},
[x_{1}^{1},\ldots,x_{1}^{n-1},x_{2}^{i}]_{k},x_{2}^{i+1},\ldots,
x_{2}^{n-1},z]_{l}.
\end{array}
\end{equation}
Set $\mathcal{O}b_{\mathcal{N}}$ to be the  right hand side of (\ref{nizar}) for  the obstruction of a deformation of a $n$-Lie algebra $\mathcal{N}$. Similarly,
set
\begin{eqnarray*}
\mathcal{O}b_{\mathcal{N}'}=&&-\sum\limits_{\begin{subarray}{l}k+l=N+1\\~~k,l>0
\end{subarray}}[y_{1}^{1},\ldots,y_{1}^{n-1},[y_{2}^{1},\ldots,y_{2}^{n-1},z']'_{k}]'_{l}
+\sum\limits_{\begin{subarray}{l}k+l=N+1\\~~k,l>0
\end{subarray}}[y_{2}^{1},\ldots,y_{2}^{n-1},[y_{1}^{1},\ldots,y_{1}^{n-1},z']'_{k}]'_{l}
\\
\nonumber
&&+
\sum\limits_{\begin{subarray}{l}k+l=N+1\\~~k,l>0
\end{subarray}}\sum\limits_{i=1}^{n-1}[y_{2}^{1},\ldots,y_{2}^{i-1},
[y_{1}^{1},\ldots,y_{1}^{n-1},y_{2}^{i}]'_{k},y_{2}^{i+1},\ldots,
y_{2}^{n-1},z']'_{l},
\end{eqnarray*}
where $y_i=(y_i^1,\ldots,y_{i}^{n-1})\in\mathcal{L}(\mathcal{N}'),~z'\in\mathcal{N}$, for  the obstruction of a deformation of a $n$-Lie algebra $\mathcal{N}'$.
On the other side, the deformation equation associated to $\phi$ is
  $\phi_{t}\circ[x_{1},\ldots,x_{n}]_{t}=[\phi_{t}(x_{1}),\ldots,\phi_{t}(x_{n})]'_{t}$
which is equivalent to
\begin{equation*}
 \sum\limits_{i+j=N+1}\phi_{i}\circ[x_{1},\ldots,x_{n}]_{j}=\sum\limits_{i_{1}+\cdots+i_{n}+j=N+1}
 [\phi_{i_{1}}(x_{1}),\ldots,\phi_{i_{n}}(x_{n})]'_{j}.
\end{equation*}
For an arbitrary $N>0$, the $(N+1)$-equation may be written as follows:
$$\delta^{1}\phi_{N+1}(x_{1},\ldots,x_{n})-\phi[x_{1},\ldots,x_{n}]_{N+1}+[\phi(x_{1}),\ldots,\phi(x_{n})]'_{N+1}
=\sum\limits_{\begin{subarray}{l}i+j=N+1\\~~i,j>0\end{subarray}}
\phi_{i}\circ[x_{1},\ldots,x_{n}]_{j}-\sum'[\phi_{i_{1}}(x_{1}),\cdots,\phi_{i_{n}}(x_{n})]_{j}$$
with
\begin{equation*}
  \sum'=\sum\limits_{j=1}^{N}\sum\limits_{\begin{subarray}{l}~~l_{i}>0\\1\leq i\leq n\end{subarray}}
 +\sum\limits_{j=1}^{N}
  \sum\limits_{\begin{subarray}{l}l_{1}+\cdots+\widehat{l}_{i}+\cdots+l_{n}>0\\~~~~~l_{i}>0\\~~~1\leq i \leq n \end{subarray}}
 +\sum\limits_{i=1}^{n}
  \sum\limits_{\begin{subarray}{l}l_{i}+\cdots+\widehat{l}_{i}+\cdots+
  l_{n}=N+1-l_{i}\\~~~~~~~~~~l_{i}>0\\~~~~~~~~ 1\leq i \leq n.\end{subarray}}.
\end{equation*}
Set $$\mathcal{O}b_{\phi}=\sum\limits_{\begin{subarray}{l}i+j=N+1\\~~i,j>0\end{subarray}}
\phi_{i}\circ[x_{1},\ldots,x_{n}]_{j}-\sum'[\phi_{i_{1}}(x_{1}),\cdots,\phi_{i_{n}}(x_{n})]'_{j}$$ for the obstruction of the extension of the
$n$-Lie algebra morphism $\phi$.
\vspace{0.3cm}

\begin{thm}
Let $(\mathcal{N},[\cdot ,\ldots,\cdot ])$ and $(\mathcal{N}',[\cdot ,\ldots,\cdot ]')$ be 
two $n$-Lie algebras and $\phi$ a $n$-Lie algebra morphism.
Let $\Theta_{t}=([\cdot ,\ldots,\cdot ]_{t},[\cdot ,\ldots,\cdot ]_{t}',\phi_{t})$
be an order $N$ one-parameter formal deformation of $\phi$.
Then
\[\mathcal{O}b=(\mathcal{O}b_{\mathcal{N}},\mathcal{O}b_{\mathcal{N}'},
\mathcal{O}b_{\phi})\in Z^{3}(\phi,\phi).
\]
Therefore the deformation extends to a deformation of order $N+1$ if and only if
$\mathcal{O}b$ is a coboundary.
\end{thm}
\begin{proof}
We must show that $(\mathcal{O}b_{\mathcal{N}},
\mathcal{O}b_{\mathcal{N}},\mathcal{O}b_{\phi})$ is
a cocycle in $C^{3}(\phi,\phi)$ i.e.
$\delta^{3}(\mathcal{O}b_{\mathcal{N}},\mathcal{O}b_{\mathcal{N}'},
\mathcal{O}b_{\phi})=0$ i.e. $(\delta^{3}\mathcal{O}b_{\mathcal{N}},
\delta^{3}\mathcal{O}b_{\mathcal{N}'},\delta^{2}\mathcal{O}b_{\phi}+
\phi\circ\mathcal{O}b_{\mathcal{N}}-\mathcal{O}b_{\mathcal{N}'}\circ(\bar{\phi}^{\otimes 2}\wedge\phi))=0$.
One has already that $\delta^{3}\mathcal{O}b_{\mathcal{N}}=0$ and
$\delta^{3}\mathcal{O}b_{\mathcal{N}'}=0$, then it remains to show
that $$\delta^{2}\mathcal{O}b_{\phi}+\phi\circ\mathcal{O}b_{\mathcal{N}}
-\mathcal{O}b_{\mathcal{N}'}\circ(\bar{\phi}^{\otimes 2}\wedge\phi)=0.$$
Let $x_{i}=x_{i}^{1}\wedge\ldots\wedge x_{i}^{n-1}\in \mathcal{L}(\mathcal{N})$
and $z\in \mathcal{N}$. We have
\[
\begin{array}{llllllllllll}
&\delta^{2}Ob_\phi(x_{1},x_{2},z)=
-\sum\limits_{i=1}^{n-1}\sum\limits_{\begin{subarray}{l}k+l=N+1\\k,l>0\end{subarray}}
\phi_{k}\circ[x_{2}^{1},\ldots,[x_{1}^{1},\ldots,x_{1}^{n-1},x_{2}^{i}],\ldots,x_{2}^{n-1},z]_{l}
-\sum\limits_{\begin{subarray}{l}k+l=N+1\\k,l>0\end{subarray}}
\phi_{k}\circ[x_{2}^{1},\ldots,x_{2}^{n-1},[x_{1}^{1},\ldots,x_{1}^{n-1},z]]_{l}\\
&
+\sum\limits_{\begin{subarray}{l}k+l=N+1\\k,l>0\end{subarray}}
\phi_{k}\circ[x_{1}^{1},\ldots,x_{1}^{n-1},[x_{2}^{1},\ldots,x_{2}^{n-1},z]]_{l}
+
\sum\limits_{\begin{subarray}{l}k+l=N+1\\k,l>0\end{subarray}}
[\phi(x_{1}^{1}),\ldots,\phi(x_{1}^{n-1}),\phi_{k}[x_{2}^{1},\ldots,x_{2}^{n-1},z]_{l}]'\\
&
-\sum\limits_{\begin{subarray}{l}k+l=N+1\\k,l>0\end{subarray}}
[\phi(x_{2}^{1}),\ldots,\phi(x_{2}^{n-1}),\phi_{k}[x_{1}^{1},\ldots,x_{1}^{n-1},z]_{l}]
-\sum\limits_{i=1}^{n-1}\sum\limits_{\begin{subarray}{l}k+l=N+1\\k,l>0\end{subarray}}
[\phi(x_{2}^{1}),\ldots,\phi_{l_{k}}[x_{1}^{1},\ldots,x_{1}^{n-1},x_{2}^{i}]_{l},\ldots,\phi(x_{2}^{n-1}),\phi(z)]'\\
&
+\sum\limits_{i=1}^{n-1}\sum'
[\phi_{l_{1}}(x_{2}^{1}),\ldots,\phi_{l_{i}}[x_{1}^{1},\ldots,x_{1}^{n-1},x_{2}^{i}],\ldots,\phi_{l_{n-1}}(x_{2}^{n-1}),\phi(z)]'_{l}\\
&
-\sum'
[\phi_{l_{1}}(x_{1}^{1}),\ldots,\phi_{l_{n-1}}(x_{1}^{n-1}),\phi_{l_{n}}[x_{2}^{1},\ldots,x_{2}^{n-1},z]]'_{j}
+\sum'
[\phi_{l_{1}}(x_{2}^{1}),\ldots,\phi_{l_{n-1}}(x_{2}^{n-1}),\phi_{l_{n}}[x_{1}^{1},\ldots,x_{1}^{n-1},z]]'_{j}\\
&
-\sum'
[\phi(x_{1}^{1}),\ldots,\phi(x_{1}^{n-1}),[\phi_{l_{1}}(x_{2}^{1}),\ldots,\phi_{l_{n-1}}(x_{2}^{n-1}),\phi_{l_{n}}(z)]'_{j}]
+\sum'
[\phi(x_{2}^{1}),\ldots,\phi(x_{2}^{n-1}),[\phi_{l_{1}}(x_{1}^{1}),\ldots,\phi_{l_{n-1}}(x_{1}^{n-1}),\phi_{l_{n}}(z)]'_{j}]'\\
&
+\sum'
[\phi(x_{2}^{1}),\ldots,[\phi_{l_{1}}(x_{1}^{1}),\ldots,\phi_{l_{n-1}}(x_{1}^{n-1}),\phi_{l_{n}}(x_{2}^{i})]'_{j},
\ldots,\phi(x_{2}^{n-1}),\phi(z)]'.
\end{array}
\]
Now, we search the terms $\phi\circ\mathcal{O}b_{\mathcal{N}}(x_1, x_2,z)$ and $\mathcal{O}b_{\mathcal{N}'}(\bar{\phi}(x_1), \bar{\phi}(x_2),\phi(z))$ in $\delta^{2}Ob_\phi$. By a straightforward but lengthy computation, we can check that  the remaining terms of
$\delta^{2}Ob_\phi+\phi\circ Ob_{\mathcal{N}}-Ob_{\mathcal{N}'}\circ(\bar{\phi}^{\otimes 2}\wedge\phi)$ are written as follows
\begin{equation}\label{eq1}
\begin{array}{llll}
&-&\widetilde{\sum}[\phi_{l_{1}}(x_{1}^{1}),\ldots,\phi_{l_{n-1}}(x_{1}^{n-1}),
[\phi_{q_{1}}(x_{2}^{1}),\ldots,\phi_{q_{n-1}}(x_{2}^{n-1}),
\phi_{q_{n}}(z)]'_{\alpha}]'_{j}\\[0.5pt]
&+&\widetilde{\sum}[\phi_{l_{1}}(x_{2}^{1}),\ldots,\phi_{l_{n-1}}(x_{2}^{n-1}),
[\phi_{q_{1}}(x_{1}^{1}),\ldots,\phi_{q_{n-1}}(x_{2}^{n-1}),\phi_{q_{n}}(z)]'_{\alpha}]'_{j}\\[0.5pt]
&+&\widetilde{\sum}\sum\limits_{i=1}^{n-1}[\phi_{l_{1}}(x_{2}^{1}),\ldots,\phi_{l_{i}}(x_{2}^{i-1}),
[\phi_{q_{1}}(x_{1}^{1}),\ldots,\phi_{q_{n-1}}(x_{1}^{n-1}),
\phi_{q_{n}}(x_{2}^{i})]_{\alpha}',\ldots,\phi_{l_{n-1}}(x_{2}^{n-1}),\phi_{l_{n}}(z)]'_{j}
\end{array}
\end{equation}
with $$\widetilde{\sum}=\sum\limits_{1\leq l_{1}+\ldots,l_{n-1}+q_{1}+\ldots+q_{n}\leq N}
+\sum\limits_{\begin{subarray}{l}l_{1}+\ldots+l_{n}=N+1\\~l_{n}=q_{1}+\ldots+q_{n}\\~~~l_{i}>0,\alpha=j=0\end{subarray}}
+\sum\limits_{\begin{subarray}{l}q_{1}+\ldots+q_{n}=N+1\\~~~~l_{i}=j=\alpha=0\\~~~~1\leq i \leq n-1\end{subarray}}.
$$
Thanks to Nambu identity, we deduce that expression (\ref{eq1}) vanishes. Thus $\mathcal{O}b\in Z^{3}(\phi,\phi)$.
One has moreover
$$\delta^{2}([\cdot ,\ldots,\cdot ]_{N+1},[\cdot ,\ldots,\cdot ]'_{N+1},\phi_{N+1})=\mathcal{O}b.$$
Then, the $N$-order formal deformation extends to a $N+1$-order formal deformation
whenever $\mathcal{O}b$ 
is a coboundary.
\end{proof}
\begin{cor}
If $H^{3}(\phi,\phi)=0$, then every infinitesimal deformation can be extended
to a formal deformation of larger order.
\end{cor}
\section{Examples}
In this section, in order to illustrate the theory, we provide some examples of  cohomology group computations and deformations. The calculations are done with the help of the computer software Mathematica.
\paragraph{Example$1.$}
We consider two $4$-dimensional $3$-Lie algebras $(A,[.,.,.]_{1})$ (resp. $ (B,[.,.,.]_{2})$) given in \cite{JAS,KA})
defined with respect to the basis $(e_{i})_{1\leq i\leq4}$ (resp. $(f_{i})_{1\leq i\leq4}$)  by
 $$[e_{1},e_{2},e_{3}]_{1}=e_{2}, \quad
[e_{1},e_{3},e_{4}]_{1}=e_{4}.$$
and
$$[f_{1},f_{2},f_{4}]_{2}=f_{3}, \quad [f_{1},f_{3},f_{4}]_{2}=f_{3}.$$
Straightforward calculation shows that  $\mathrm{dim} H^{2}(A,A)=3$
and spanned by the 2-cocycles
$$
\left\{\begin{array}{llll}
\psi_{1,1}(e_{1},e_{2},e_{3})=e_{1}-e_{3}\\
\psi_{1,1}(e_{1},e_{2},e_{4})=e_{4}\\
\psi_{1,1}(e_{1},e_{3},e_{4})=0\\
\psi_{1,1,}(e_{2},e_{3},e_{4})=e_{4},
\end{array}\right.
\quad
\left\{\begin{array}{llll}
\psi_{1,2}(e_{1},e_{2},e_{3})=e_{2}\\
\psi_{1,2}(e_{1},e_{2},e_{4})=0\\
\psi_{1,2}(e_{1},e_{3},e_{4})=0\\
\psi_{1,2}(e_{2},e_{3},e_{4})=0,
\end{array}\right.
\quad
\left\{\begin{array}{llll}
\psi_{1,3}(e_{1},e_{2},e_{3})=0\\
\psi_{1,3}(e_{1},e_{2},e_{4})=e_{2}\\
\psi_{1,3}(e_{1},e_{3},e_{4})=e_{1}-e_{3}\\
\psi_{1,3}(e_{2},e_{3},e_{4})=e_{2}.
\end{array}\right.
$$
Similarly, we  show that 
$H^{2}(B,B)$ is $2$-dimensional and spanned by the $2$-cocycles
$$
\left\{\begin{array}{llll}
\psi_{2,1}(f_{1},f_{2},f_{3})=0\\
\psi_{2,1}(f_{1},f_{2},f_{4})=f_{2}\\
\psi_{2,1}(f_{1},f_{3},f_{4})=0\\
\psi_{2,1}(f_{2},f_{3},f_{4})=0,
\end{array}\right.
\quad\hbox{ and }\quad
\left\{\begin{array}{llll}
\psi_{2,2}(f_{1},f_{2},f_{3})=0\\
\psi_{2,2}(f_{1},f_{2},f_{4})=0\\
\psi_{2,2}(f_{1},f_{3},f_{4})=f_{2}\\
\psi_{2,2}(f_{2},f_{3},f_{4})=0.
\end{array}\right.
$$
Now, we compute first all  the $3$-Lie algebra morphisms $\phi_{1,2}:A\rightarrow B$.
The morphism $\phi_{1,2}$ is wholly determined by a set of structure constants $\lambda_{i,j}$,
where 
$\phi_{1,2}(e_{j})=\sum\limits_{j=1}^{4}\lambda_{i,j}f_{i}.$
It turns out that they are defined by\\
$$\small{\left\{\begin{array}{ll}
\phi_{1,2}(e_{1})=\lambda_{1,1}f_{1}+\lambda_{2,1}f_{2}+\lambda_{3,1}f_{3}+\lambda_{4,1}f_{4}\\
\phi_{1,2}(e_{2})=0\\
\phi_{1,2}(e_{3})=\lambda_{3,1}f_{1}+\lambda_{2,3}f_{2}+\lambda_{3,3}f_{3}+\lambda_{4,3}f_{4}\\
\phi_{1,2}(e_{4})=0
\end{array}\right.~~ \hbox{ or } ~~\left\{\begin{array}{ll}
\phi_{1,2}(e_{1})=\lambda_{1,1}f_{1}+\lambda_{2,1}f_{2}+\lambda_{3,1}f_{3}+\lambda_{4,1}f_{4}\\
\phi_{1,2}(e_{2})=0\\
\phi_{1,2}(e_{3})=\lambda_{1,3}f_{1}+\lambda_{2,3}f_{2}+\lambda_{3,3}f_{3}+\frac{\lambda_{1,3}
\lambda_{4,1}-1}{\lambda_{1,1}}f_{4}\\
\phi_{1,2}(e_{4})=\lambda_{3,4}f_{3}
\end{array}\right.}$$

$$\small{\hbox{ or }~~\left\{\begin{array}{ll}
\phi_{1,2}(e_{1})=\lambda_{1,1}f_{1}+\lambda_{2,1}f_{2}+\lambda_{3,1}f_{3}+\lambda_{4,1}f_{4}\\
\phi_{1,2}(e_{2})=\lambda_{3,2}f_{3}\\
\phi_{1,2}(e_{3})=\lambda_{1,3}f_{1}+\lambda_{2,3}f_{2}+\lambda_{3,3}f_{3}+
\frac{\lambda_{1,3}\lambda_{4,1}+1}{\lambda_{1,1}}f_{4}\\
\phi_{1,2}(e_{4})=0
\end{array}\right. ~ ~\hbox{ or }~~
\left\{\begin{array}{ll}
\phi_{1,2}(e_{1})=\lambda_{2,1}f_{2}+\lambda_{3,1}f_{3}+\lambda_{4,1}f_{4}\\
\phi_{1,2}(e_{2})=\lambda_{3,2}f_{3}\\
\phi_{1,2}(e_{3})=-\frac{1}{\lambda_{4,1}}f_{1}+\lambda_{2,3}f_{2}+\lambda_{3,3}f_{3}+\lambda_{4,3}f_{4}\\
\phi_{1,2}(e_{4})=0
\end{array}\right.} $$
$$\small{\hbox{ or }~~\left\{\begin{array}{ll}
\phi_{1,2}(e_{1})=\lambda_{2,1}f_{2}+\lambda_{3,1}f_{3}+\lambda_{4,1}f_{4}\\
\phi_{1,2}(e_{2})=0\\
\phi_{1,2}(e_{3})=\frac{1}{\lambda_{4,1}}f_{1}+\lambda_{2,3}f_{2}+\lambda_{3,3}f_{3}+\lambda_{4,3}f_{4}\\
\phi_{1,2}(e_{4})=\lambda_{3,4}f_{3}
\end{array}\right.} $$
By a direct computation, using Mathematica, we deduce  that
the first space of cocycles
$Z^{1}(A,B)$ of the first morphism with
$\phi_{1,2}(e_{1})=\phi_{1,2}(e_{3})$ is generated by
$$\left\{\begin{array}{llllll}
\rho_{k}(e_{2})=\rho_{k}(e_{4})=0,&\hbox{ for } k\in\{1,\ldots,8\},\\
\rho_{k}(e_{3})=0, ~~\rho_{k}(e_{1})=f_{k},&\hbox{ for } k\in\{1,\ldots,4\},\\
\rho_{k}(e_{1})=0, ~~\rho_{k}(e_{3})=f_{k-4},&\hbox{ for } k\in\{5,\ldots,8\}.
\end{array}\right.$$
Moreover, we have  $H^{1}(A,B)$ is $8$-dimensional.
\paragraph{Example $2.$}
In this example, we consider the $3$-Lie algebra $A$ of the previous example and $B$ defined as 
$$ [f_{1},f_{3},f_{4}]_{2}=f_{2}.$$
By a direct computation, we obtain
$\mathrm{dim} H^{2}(B,B)=3$
and it is spanned by the $2$-cocycles
$$
\left\{\begin{array}{llll}
\psi_{1,1}(f_{1},f_{2},f_{3})=0\\
\psi_{1,1}(f_{1},f_{2},f_{4})=0\\
\psi_{1,1}(f_{1},f_{3},f_{4})=f_{1}\\
\psi_{1,1,}(f_{2},f_{3},f_{4})=0,
\end{array}\right.
\quad
\left\{\begin{array}{llll}
\psi_{1,2}(f_{1},f_{2},f_{3})=0\\
\psi_{1,2}(f_{1},f_{2},f_{4})=0\\
\psi_{1,2}(f_{1},f_{3},f_{4})=f_{3}-f_{4}\\
\psi_{1,2}(f_{2},f_{3},f_{4})=0,
\end{array}\right.
\quad
\left\{\begin{array}{llll}
\psi_{1,3}(f_{1},f_{2},f_{3})=0\\
\psi_{1,3}(f_{1},f_{2},f_{4})=0\\
\psi_{1,3}(f_{1},f_{3},f_{4})=0\\
\psi_{1,3}(f_{2},f_{3},f_{4})=f_{1}.
\end{array}\right.
$$
Now, we consider  the $3$-Lie algebra morphism  defined as
$$\small{\left\{\begin{array}{lllll}
\phi(e_{1})=\lambda_{1,1}f_{1}+\lambda_{2,1}f_{2}+\lambda_{3,1}f_{3}+\lambda_{4,1}f_{4}\\
\phi(e_{3})=\lambda_{1,3}f_{1}+\lambda_{2,3}f_{2}+\lambda_{3,3}f_{3}+\lambda_{4,3}f_{4}\\
\phi(e_{2})=\phi(e_{4})=0.
\end{array}\right.}
$$
By a direct computation, using  we deduce that
the space of 1-cocycles $Z^{1}(A,B)$ related to the morphism $\phi$ is generated by
$$\left\{\begin{array}{llllll}
\rho_{k}(e_{1})=f_{k},&\hbox{ for } k\in\{1,\ldots,4\},&\hbox{ and for } k\in I\setminus \{1,\ldots,4\}&~~~\rho_{k}(e_{1})=0 \\
\rho_{k}(e_{2})=f_{k-4},&\hbox{ for } k\in\{5,\ldots,8\},&\hbox{ and for } k\in I\setminus \{5,\ldots,8\}&~~~\rho_{k}(e_{2})=0 \\
\rho_{k}(e_{3})=f_{k-8},&\hbox{ for } k\in\{9,\ldots,12\},&\hbox{ and for } k\in I\setminus \{9,\ldots,12\}&~~~\rho_{k}(e_{3})=0 \\
\rho_{k}(e_{4})=f_{k-12},&\hbox{ for } k\in\{13,\ldots,16\},&\hbox{ and for } k\in I\setminus \{13,\ldots,16\}&~~~\rho_{k}(e_{4})=0, \\
\end{array}\right.$$
where  $I=\{1,\ldots,16\}$. Hence $H^{1}(A,B)$ is $16$-dimensional.
\paragraph{Example $3.$}
We consider two $4$-dimensional $3$-Lie algebras $(A,[.,.,.]_{A})$ ~$(resp. (B,[.,.,.]_{B}))$
defined in \cite{RGY} with respect to the basis $(e_{i})_{1\leq i\leq4}$ (resp. $(f_{i})_{1\leq i\leq4}$)  by
$$[e_{2}, e_{3},e_{4}]_{A} = e_{1}
(\text{ resp. } \;  [f_{2},f_{3},f_{4}]_{B}=f_{1}, [f_{1},f_{3},f_{4}]_{B}=f_{2},[f_{1},f_{2},f_{4}]_{B}=f_{3},
[f_{1},f_{2},f_{3}]_{B}=f_{4}).$$
We have 
$\mathrm{dim} H^{2}(A, A)=9$
and the space is spanned by the $2$-cocycles (here $I=\{1,\ldots,6\}$)
$$
\left\{\begin{array}{llllll}
\psi_{1,k}(e_{1},e_{2},e_{4})=e_{k},&\hbox{ for } k=3,&\hbox{ and for } k\in I\setminus \{2\}&\psi_{1,k}(e_{1},e_{2},e_{4})=0 \\
\psi_{1,k}(e_{1},e_{3},e_{4})=e_{k},&\hbox{ for } k=1,2,&\hbox{ and for } k\in I\setminus \{1,2\}&\psi_{1,k}(e_{1},e_{3},e_{4})=0 \\
\psi_{1,k}(e_{2},e_{3},e_{4})=e_{k-1},&\hbox{ for } k=4,&\hbox{ and for } k\in I\setminus \{4\}&\psi_{1,k}(e_{2},e_{3},e_{4})=0 \\
\psi_{1,k}(e_{1},e_{2},e_{3})=(6-k)e_{k-4}+(k-5)e_{k-2},&\hbox{ for } k=5,6,&\hbox{ and for } k\in I\setminus \{5,6\}&\psi_{1,2}(e_{1},e_{2},e_{3})=0,
\end{array}\right.
$$
$$
\left\{\begin{array}{llll}
\psi_{1,7}(e_{1},e_{2},e_{4})=e_{2}\\
\psi_{1,7}(e_{1},e_{3},e_{4})=-e_{3}\\
\psi_{1,7}(e_{2},e_{3},e_{4})=0\\
\psi_{1,7}(e_{1},e_{2},e_{3})=0
\end{array}\right.
\quad
\left\{\begin{array}{llll}
\psi_{1,8}(e_{1},e_{2},e_{4})=0\\
\psi_{1,8}(e_{1},e_{3},e_{4})=e_{4}\\
\psi_{1,8}(e_{2},e_{3},e_{4})=0\\
\psi_{1,8}(e_{1},e_{2},e_{3})=e_{2}
\end{array}\right.\quad
\left\{\begin{array}{llll}
\psi_{1,9}(e_{1},e_{2},e_{4})=-e_{4}\\
\psi_{1,9}(e_{1},e_{3},e_{4})=0\\
\psi_{1,9}(e_{2},e_{3},e_{4})=0\\
\psi_{1,9}(e_{1},e_{2},e_{3})=e_{3}
\end{array}\right.
$$
One can check that any 2-cocycle of the 
$3$-Lie algebra $B$ is a coboundary, 
hence $\mathrm{dim} H^{2}(B, B)=0$.\\
Now,  we  construct the $3$-Lie algebra morphisms $\phi: A\rightarrow B$. It  turns out that they are defined as
$$\small{\left\{\begin{array}{ll}
\phi(e_{1})=\phi(e_{2})=0,\\
\phi(e_{3})=\lambda_{4,3}f_{4}\\
\phi(e_{4})=\lambda_{2,4}f_{2}+\lambda_{2,4}f_{3}+\lambda_{4,4}f_{4}.
\end{array}\right.}
$$
Similarly, one  checks
that $H^{1}(A,B)$ is spanned by (here   $I=\{1,\ldots,16\}$)
$$\left\{\begin{array}{llllll}
\rho_{k}(e_{1})=f_{k},&\hbox{ for } k\in\{2,\ldots,4\},&\hbox{ and for } k\in I\setminus \{2,\ldots,4\}&~~~\rho_{k}(e_{1})=0 \\
\rho_{k}(e_{2})=f_{k-4},&\hbox{ for } k\in\{5,\ldots,8\},&\hbox{ and for } k\in I\setminus \{5,\ldots,8\}&~~~\rho_{k}(e_{2})=0 \\
\rho_{k}(e_{3})=f_{k-8},&\hbox{ for } k\in\{9,\ldots,12\},&\hbox{ and for } k\in I\setminus \{9,\ldots,12\}&~~~\rho_{k}(e_{3})=0 \\
\rho_{k}(e_{4})=f_{k-12},&\hbox{ for } k\in\{13,\ldots,16\},&\hbox{ and for } k\in I\setminus \{13,\ldots,16\}&~~~\rho_{k}(e_{4})=0. \\
\end{array}\right.$$
In the following, we deal with deformations. We consider two infinitesimal deformations of $[.,.,.]_{A}$:
$$
\begin{array}{lllll}
[e_{2},e_{3},e_{4}]_{A,1,t}=e_{1}+tc_{2}e_{2},\quad~~ [e_{1},e_{2},e_{4}]_{A,1,t}=tz_{2}e_{2},\quad~~
[e_{1},e_{3},e_{4}]_{A,1,t}=t(b_{2}e_{2}-z_{2}e_{3}),\quad~~ [e_{1},e_{2},e_{3}]_{A,1,t}=0,\quad\quad\quad\quad\quad\hfill \\[10pt]
[e_{2},e_{3},e_{4}]_{A,2,t}=e_{1}+tc_{4}e_{4},\quad~~ [e_{1},e_{2},e_{4}]_{A, 2,t}=tz_{2}e_{2},\quad~~
[e_{1},e_{3},e_{4}]_{A,2,t}=-tz_{2}e_{3},\quad~~ [e_{1},e_{2},e_{3}]_{A,2,t}=tk_{4}e_{4}.\quad\quad\quad\quad\quad\quad \quad\hfill
\end{array}
$$
Then, we have  three infinitesimal deformations $([.,.,.]^i_{A,t},[.,.,.]_{B},\phi^i_{t})$ of $\phi$ given by
$$
\begin{array}{llllll}
\left\{\begin{array}{ll}
\phi^1_{t}(e_{1})=\phi^1_{t}(e_{2})=0\\
\phi^1_{t}(e_{3})=(\lambda_{4,3}+tb_{4,3})f_{4}\\
\phi^1_{t}(e_{4})=(\lambda_{2,4}+tb_{2,4})f_{2}+\lambda_{2,4}f_{3}+\lambda_{4,4}f_{4}
\end{array}\right.
&\hbox{ and }~~\left\{\begin{array}{llll}
\  [e_{2},e_{3},e_{4}]^1_{A,t}=e_{1}+tc_{2}e_{2}\\
 \  [e_{1},e_{2},e_{4}]^1_{A,t}=[e_{1},e_{3},e_{4}]^1_{A,t}=0\\
\ [e_{1},e_{2},e_{3}]^1_{A,t}=0
\end{array}\right.\\[20pt]
\left\{\begin{array}{ll}
\phi^2_{t}(e_{1})=c_{4}t(-\lambda_{2,4}f_{2}-\lambda_{2,4}f_{3}-\lambda_{4,4}f_{4})\\
\phi^2_{t}(e_{2})=t(b_{2,2}f_{2}+b_{2,2}f_{3}+b_{4,2}f_{4})\\
\phi^2_{t}(e_{3})=t(b_{1,3}f_{1}+b_{2,3}f_{2}+b_{3,3}f_{3})+(\lambda_{4,3}tb_{4,3})f_{4}\\
\phi^2_{t}(e_{4})=tb_{1,4}f_{1}+(\lambda_{2,4}+tb_{2,4})f_{2}+(\lambda_{2,4}+tb_{2,4})f_{3}+
(\lambda_{4,4}+tb_{4,4})f_{4}\end{array}\right.
&\hbox{ and }~~\left\{\begin{array}{llll}
\ [e_{2},e_{3},e_{4}]^2_{A,t}=e_{1}+tc_{4}e_{4}\\
\ [e_{1},e_{2},e_{4}]^2_{A,t}=[e_{1},e_{3},e_{4}]^2_{A,t}=0\\
\ [e_{1},e_{2},e_{3}]^2_{A,t}=0
\end{array}\right.\\[20pt]
\left\{\begin{array}{ll}
\phi^3_{t}(e_{1})=0\\
\phi^3_{t}(e_{2})=t(b_{2,2}f_{2}+b_{2,2}f_{3}+b_{4,2}f_{4})\\
\phi^3_{t}(e_{3})=t(b_{1,3}f_{1}+b_{2,3}f_{2}+b_{3,3}f_{3})+(\lambda_{4,3}+tb_{4,3})f_{4}\\
\phi^3_{t}(e_{4})=tb_{1,4}f_{1}+(\lambda_{2,4}+tb_{2,4})f_{2}+(\lambda_{2,4}+tb_{3,4})f_{3}+
(\lambda_{4,4}+tb_{4,4})f_{4}\end{array}\right.
&\hbox{ and }~~\left\{\begin{array}{llll}
\ [e_{2},e_{3},e_{4}]^3_{A,t}=e_{1}\\
\ [e_{1},e_{3},e_{4}]^3_{A,t}=tb_{2}e_{2}\\
\ [e_{1},e_{2},e_{4}]^3_{A,t}=[e_{1},e_{2},e_{3}]^3_{A,t}=0.
\end{array}\right.
\end{array}
$$
Finally, we will construct a formal automorphism $\psi_{A,t}: A[[t]]\rightarrow A[[t]]$ modulo $t^{2}$. By a direct computation, we can  see that 
such  automorphism modulo $t^{2}$  is
$\psi_{A,t}=I_{A}+t
\varphi_{A}$  and $\psi_{A,t}^{-1}=I_{A}-t\varphi_{A}$, where $\varphi_A\in C^{1}(A,A)$ is defined as follows 
$$\left\{\begin{array}{ll}
\varphi_{A}(e_{1})=(b'_{2,2}+b'_{3,3}+b'_{4,4})e_{1}\\
\varphi_{A}(e_{2})=b'_{1,2}e_{1}+b'_{2,2}e_{2}+b'_{3,2}e_{3}+b'_{4,2}e_{4}\\
\varphi_{A}(e_{3})=b'_{1,3}e_{1}+b'_{2,3}e_{2}+b'_{3,3}e_{3}+b'_{4,3}e_{4}\\
\varphi_{A}(e_{4})=b'_{1,4}e_{1}+b'_{2,4}e_{2}+b'_{3,4}e_{3}+b'_{4,4}e_{4}.
\end{array}
\right.
$$
Setting $\psi_{B,t}=I_{B[[t]]}$ and 
$\widetilde{\phi}^1_{t}=\psi_{B,t}\circ\phi^1_t\circ\psi_{A,t}^{-1}$, we get
$$\left\{\begin{array}{ll}
\widetilde{\phi}^{1}_{t}(e_{1})=\widetilde{\phi}(e_{2})=0\\
\widetilde{\phi}^{1}_{t}(e_{3})=-t(b'_{4,3}\lambda_{2,4}f_{2}+b'_{4,3}\lambda_{2,4}f_{3})+
(\lambda_{4,3}+t(b_{4,3}-\lambda_{4,3}b'_{3,3}-b'_{4,3}\lambda_{4,4}))f_{4}\\
\widetilde{\phi}^{1}_{t}(e_{4})=(\lambda_{2,4}+t(b_{2,4}-b'_{4,4}\lambda_{2,4}))f_{2}+
(\lambda_{2,4}-tb'_{4,4}\lambda_{2,4})f_{3}+(\lambda_{4,4}-t(b'_{3,4}\lambda_{4,3}+b'_{4,4}\lambda_{4,4}))f_{4}.
\end{array}
\right.
$$
Moreover, setting $[.,.,.]_{A,t}^{1'}=\psi_{A,t}\circ[.,.,.]^1_{A,t}\circ(\psi^{-1}_{A,t})^{\otimes 3}$, we get
$$\left\{\begin{array}{ll}
\ [e_{2},e_{3},e_{4}]'_{A,t}=(1+(b'_{3,3}+b'_{4,4})t)e_{1}+c_{2}te_{2}\\
\ [e_{1},e_{2},e_{4}]'_{A,t}=[e_{1},e_{3},e_{4}]'_{A,t}=[e_{1},e_{2},e_{3}]'_{A,t}=0.
\end{array}\right.
$$
Thus, we deduce that $([.,.,.]_{A,t}^{1'},[.,.,.]_{B},\widetilde{\phi}^1_{t})$ is an equivalent infinitesimal deformation of $([.,.,.]^1_{A,t},[.,.,.]_{B},{\phi}^1_{t})$. 
We can construct other  examples of equivalent infinitesimal deformations.

\paragraph{Conclusion.}
We have performed some computations of cohomology groups for 3-Lie algebra morphisms and applications to deformation theory. This is in order to illustrate the new cohomology of 3-Lie algebra morphisms introduced in this paper.


\end{document}